\documentclass[smallextended]{svjour3}
\usepackage{bbm,cite}
\usepackage[numbers]{natbib}
\usepackage[colorlinks,citecolor=blue,linkcolor=blue,urlcolor=blue]{hyperref}

\usepackage{latexsym, epsfig, epstopdf, amssymb, amsmath, amsthm, graphicx, mathrsfs, booktabs, mathtools}
\usepackage{rotating, tabularx, setspace,paralist}
\usepackage{color}
\usepackage{lscape}
\usepackage{multirow,multicol}
\usepackage{algorithm}
\usepackage{algpseudocode}
\usepackage{anysize}
\usepackage{siunitx}
\usepackage{comment}
\usepackage{todonotes}

\newcommand{\ignore}[1]{}{}

\newcommand{\gap}{\vspace{0.1in}}

\newcommand{\epc}{\hspace{1pc}}
\newcommand{\thalf}{{\textstyle{\frac{1}{2}}}}

\newcommand{\onebld}{{\bf 1}}

\newcommand{\wh}{\widehat}

\let\svthefootnote\thefootnote
\newcommand\freefootnote[1]{%
  \let\thefootnote\relax%
  \footnotetext{#1}%
  \let\thefootnote\svthefootnote%
}

\makeatletter
\newcommand\blfootnote[1]{%
  \begingroup
  \renewcommand{\@makefntext}[1]{\noindent\makebox[1.8em][r]#1}
  \renewcommand\thefootnote{}\footnote{#1}%
  \addtocounter{footnote}{-1}%
  \endgroup
}
\makeatother

\usepackage{stackengine,scalerel}

\DeclareMathOperator*{\proj}{proj}
\DeclareMathOperator{\conv}{conv}

\DeclareMathOperator*{\argmin}{arg\,min}

\newcommand{\R}{\mathbb{R}}

\usepackage{xcolor}

\title{
Improving the Solution of Indefinite Quadratic Programs 
and Linear Programs with Complementarity Constraints 
by a Progressive MIP Method}

\titlerunning{Improving the Solution of Indefinite QP and LPCC by a 
Progressive MIP Method}

\author{Xinyao Zhang \and Shaoning Han \and Jong-Shi Pang\thanks{This work is based on researched supported by the Air Force Office of Sponsored Research under grant FA9550-22-1-0045.}}
 
\institute{
    Xinyao Zhang and Jong-Shi Pang \at
    The Daniel J.\ Epstein Department of Industrial and Systems Engineering, \\
    University of Southern California, \\
    Los Angeles, CA 90066, USA \\
    \email{xinayoz@usc.edu; jongship@usc.edu}
    \and
    Shaoning Han \at
    Department of Mathematics, \\
    Institute of Operations Research and Analytics,\\
    National University of Singapore, \\
    Singapore 119076 \\
    \email{shaoninghan@nus.edu.sg}
}

\date{Original August 2024; revised February 2025}

\begin{document}

\maketitle  

\begin{abstract}
\noindent Indefinite quadratic programs (QPs) 
are known to be very difficult to be solved 
to global optimality, so are linear programs 
with linear complementarity constraints 
(LPCCs).  It is a classic 
result that for a QP with an optimal solution, 
the QP has an equivalent formulation as a 
certain LPCC in terms of their globally
optimal solutions.   Thus it is natural to
attempt to solve an (indefinite) QP as a 
LPCC.  This paper presents a 
progressive mixed integer linear programming 
method for solving a general LPCC. Instead 
of solving the LPCC with a full 
set of integer variables expressing
the complementarity conditions, the presented
method solves 
a finite number of mixed integer subprograms 
by starting with a small fraction of integer 
variables and 
progressively increasing this fraction.  After 
describing 
the PIP (for progressive integer programming) 
method and providing some details 
for its implementation and tuning 
possibilities, we demonstrate, via an
extensive set of computational experiments, 
the superior performance of the progressive 
approach over the direct solution of the 
full-integer formulation of the LPCCs in obtaining high-quality solutions.  It is 
also shown that the solution obtained at the 
termination of the 
PIP method is a local minimizer of the LPCC, 
a property that 
cannot be claimed by any known non-enumerative
method for solving this nonconvex program.   
In all the experiments, the PIP method is 
initiated at a feasible solution 
of the LPCC obtained from a nonlinear 
programming solver, 
and with high likelihood, can successfully 
improve it.  Thus,
the PIP method can improve a stationary 
solution of an 
indefinite QP, something that is not likely to 
be achievable by 
a nonlinear programming method.  Finally, some 
analysis is
presented that provides a better understanding 
of the roles of the LPCC suboptimal solutions 
in the local optimality
of the indefinite QP.  This 
local aspect of the connection between a QP 
and its LPCC formulation has seemingly
not been addressed in the literature.
\end{abstract}

\keywords{Nonconvex quadratic programming 
\and Linear programs with complementarity 
constraints \and Improving stationary 
solutions \and Progressive integer programming}

\subclass{90C20 \and 90C26 \and 90C33}

\section{Introduction} \label{sec:introduction}

It is well known that indefinite quadratic 
programs (QPs) are NP-hard \cite{Vavasis90}; 
thus are very difficult to be solved
to  global optimality in practice.   This is true 
even when the Hessian of the quadratic objective 
has only a single negative 
eigenvalue \cite{PardalosVavasis91}.  The special 
case of the ``standard QP'', i.e., the case where 
the feasible region
is the unit simplex in an Euclidean space, has 
received considerable attention, partly due to its
connection to 
the ``copositivity test'' problem on the 
nonnegative orthant; see for instance 
\cite{Bomze98,ChenBurer12,Nowak99,
XiaVeraZuluaga20}.  
Section~5.2 of \cite{CuiChangHongPang20}
shows that a special homogeneous QP with one negative eigenvalue 
can be globally resolved by solving two convex
quadratic programs and the procedure can be 
extended to a QP with two negative eigenvalues 
by the use of parametric
convex quadratic programming 
\cite{CottlePangStone92} via its 
linear complementarity formulation.  Among the most
challenging QPs is the quadratic assignment 
problem; see 
\cite{BurkardKarischRendl97,AnstreicherBrixus01a,
AnstreicherBrixus01b,
AnstreicherBrixusGouxLinderoth02} 
for some references and the QAPLIB reference page 
\url{https://coral.ise.lehigh.edu/data-sets/qaplib/} 
for an extensive literature.

\gap

It is long known \cite{GiannessiTomasin73} that a solvable 
quadratic program is equivalent, in terms of globally 
optimal solutions, to a linear program with linear 
complementarity constraints (LPCC).  This classic result 
was extended in \cite{HuMitchellPang12} to a feasible QP 
that is not known a priori to have an optimal solution.  
Besides indefinite QPs, the LPCC has broad applications 
and been studied extensively; see
\cite{HuMitchellPangBennetKunapuli08,HuMitchellPangYu12, 
JMMPangWachter20}; 
see also \cite[Section~4.1]{Pang10}.  Among the special 
instances of the LPCC are those derived from 
bilevel linear programs for which there is a vast 
literature such as \cite{AudelSavardZghal07,JSRFaustino06}.
In general, using a big-M quantity to
express the complementarity condition by binary variables, 
it is well known that an LPCC admits a mixed integer 
linear program (MILP) formulation, thus enabling the global 
solution of an LPCC, and in particular, an indefinite QP, 
by MILP methods.   Nevertheless, in spite of the research 
in the cited references and the more recent work
\cite{YuMitchellPang19}, the global solution of LPCCs, 
particularly those of large dimensions, by MILP methods 
remains a daunting task.   In contrast to this aim of 
computing globally
optimal solutions, stationary solutions of an indefinite 
quadratic program can be computed by nonlinear 
programming (NLP) methods; so can those of an LPCC 
formulated as a nonlinear program.    Indeed,
there is a vast literature on nonlinear programming based 
local search methods for computing various kinds 
of stationary solutions of an LPCC; some of these methods 
have been implemented in softwares such as
{\sc knitro} \cite{ByrdHribarNocedal99} and {\sc Filter} 
\cite{FletcherLeyffer04}. 
The Ph.D.\ dissertation \cite{Jara-Moroni18}  and the 
subsequent papers 
\cite{FangLeyfferMunson20,JMMPangWachter20} are several
recent additions to this topic, 
providing a comprehensive state-of-the-art summary of the 
solution of the LPCC together with many references.

\gap

Prompted by a comment of 
a referee, we find it useful to emphasize that
we have bypassed the use of the term 
``B-stationarity'' for the LPCC throughout the 
paper (see \cite{PangFukushima99} for a brief 
summary of the origin of the 
term).  The latter term is generally used 
in the context of problems with 
nonlinear functions.  For the LPCC, it is 
known since
\cite{LuoPangRalph96} that every B-stationary 
solution must be a local minimizer, which of 
course is a sharper kind of solution.
Thus, as in \cite{JMoroniPangWaechter18}, we 
use the term ``local minimizer'' instead of 
B-stationary solution for the LPCC.

\gap

There are 
essentially two ways to attempt to improve the 
objective value of a computed solution of an LPCC
(typically of the stationary kind that is not
necessarily locally minimizing): one is to restart the 
same local search method at a different initial point, 
hoping that the final solution will be
better, or initiate another search method (perhaps with 
some randomization) at the computed solution; the other 
is to resort to a global method via the MILP formulation 
of the equivalent LPCC formulation.  A main objective of 
the present work is to advance the solution methodology of 
an LPCC, and thus an indefinite QP, by a novel 
employment of the MILP methodology, 
which we term the {\sl progressive integer 
programming} (PIP) method.   A major distinction of 
the PIP method is that while it is IP based, 
it completely bypasses the solution of the 
full-integer formulation
of the LPCC.  Instead, it solves a finite 
number of mixed integer subprograms by 
starting with a small fraction 
of integer variables and progressively 
increasing this fraction; the solution of 
these smaller sized IPs is totally
left to an integer programming solver, such 
as the most well-known GUROBI, without any 
attempt to refine
the IP methods.   Through an extensive set of 
computational experiments, we demonstrate the 
major practical benefits of the PIP method, 
which we summarize below:

\gap

\noindent (i) starting from a computed solution of the 
LPCC by a NLP solver, the PIP method consistently
improves the solution considerably without solving the full-integer LPCC formulation;

\gap

\noindent (ii) for problems for which the full-integer 
formulation of the LPCC can be solved to global 
optimality, the PIP solution matches the global solution 
either exactly, or inexactly with a small gap, and
most importantly, requiring only a fraction of 
computational time; 

\gap

\noindent (iii) the PIP method can flexibly control the 
number of integer variables in the subproblems
being solved, thus providing a practical balance 
between solution quality and computational cost,

\gap

\noindent (iv) PIP scales well to large-sized 
problems, obtaining solutions (short of
a certificate of global optimality) of superior 
quality that otherwise cannot be obtained by NLP 
methods as implemented by state-of-the-art solvers;

\gap

\noindent (v) bridging NLP and MILP, the PIP 
method can take advantage of any advances from 
either field to improve its efficiency.

\gap

\noindent 
Theoretically, when the subproblems are solved to
global optimality, the solution obtained 
by PIP is always a local minimizer of the 
LPCC, a property that cannot be claimed 
when the solution of the full-integer LPCC is 
early terminated 
by an integer programming method.  
Practically, the computational success of PIP can be 
attributed to two strategies: it warm starts the solution 
of each mixed integer subproblem from that of the
previous subproblem; it early terminates the 
solution of the subproblems without sacrificing the 
objective quality of the solution at termination.
In the case of an indefinite QP, in spite of its 
equivalence 
to the LPCC in terms of global optimality, a 
local minimizer of the latter 
problem is not necessarily a locally optimal 
solution to the given QP.  A theoretical contribution
of this paper is that we show that the solution obtained
by solving a reduced-integer formulation 
of the LPCC is globally optimal to the QP over a certain 
polyhedral subset of the feasible region.  Moreover, it 
is possible to provide a constructive sufficient condition 
for the PIP solution to be a local minimizer of the given 
QP, by solving a MILP with a reduced number of 
integer variables.

\section{The LPCC and its Connection to QPs} \label{sec:LPCC}

In its general form, a linear program with linear 
complementarity constraints (LPCC) is to find vectors
$(x,y,z) \in \mathbb{R}^{n+2m}$ to
\begin{equation} \label{eq:LPCC}
\begin{array}{ll}
\displaystyle{
\operatornamewithlimits{\mbox{\bf minimize}}_{(x,y,z) 
\in \mathbb{R}^{n + 2m}}
} & \theta \, \triangleq \, 
c^{\top}x + e^{\top}y + f^{\top}z \\ [0.1in]
\mbox{\bf subject to} & Ax + By + Cz \, = \, b \\ [0.1in]
\mbox{\bf and} & 0 \, \leq \, y \, \perp \, z \, \geq \, 0,
\end{array} \end{equation}
where $c$ is an $n$-vector, $e$ and $f$ are $m$-vectors, $y$ 
and $z$ are $m$-vectors,
$A$ is a  $k \times n$ matrix, $B$ and $C$ are $k \times m$ 
matrices, $b \in \mathbb{R}^k$,
and $\perp$ is the perpendicularity notation which in this 
context denotes the complementarity condition between the 
two vectors $y$ and $z$.   The formulation (\ref{eq:LPCC})
includes the case where the ``design'' variable $x$ is 
subject to its own linear constraints:
\begin{equation} \label{eq:LPCC with x constraint}
\begin{array}{ll}
\displaystyle{
\operatornamewithlimits{\mbox{\bf minimize}}_{(x,y,z) 
\in \mathbb{R}^{n + 2m}}
} & \theta \, \triangleq \, 
c^{\top}x + e^{\top}y + f^{\top}z \\ [0.1in]
\mbox{\bf subject to} & Ax + By + Cz \, = \, b; \
Ex \, \geq \, g \\ [0.1in]
\mbox{\bf and} & 0 \, \leq \, y \, \perp \, z \, \geq \, 0,
\end{array} \end{equation}
with $E \in \mathbb{R}^{j \times n}$ and 
$g \in \mathbb{R}^j$.
It is obvious that (\ref{eq:LPCC with x constraint}) can 
be cast
in the form (\ref{eq:LPCC}) with the introduction of the
nonnegative slack variable $s \triangleq Ex - g \geq 0$
and with the augmentations:
\[
\wh{x} \, \triangleq \, \left( \begin{array}{c}
x \\ [5pt]
s
\end{array} \right); \ 
\wh{c} \, \triangleq \, \left( \begin{array}{c}
c \\ [5pt]
0
\end{array} \right); \
\wh{A} \, \triangleq \, \left[ \begin{array}{ll}
A & \ 0 \\ [5pt]
E & \ \mathbb{I}
\end{array} \right]; \ 
\wh{B} \, \triangleq \, \left[ \begin{array}{l}
B \\ [5pt]
0
\end{array} \right]; \ 
\wh{C} \, \triangleq \, \left[ \begin{array}{l}
C \\ [5pt]
0
\end{array} \right]; \ \mbox{ and } \
\wh{b} \, \triangleq \, \left[ \begin{array}{l}
b \\ [5pt]
g
\end{array} \right],
\]
where $\mathbb{I}$ is the identity matrix, resulting in
\begin{equation} \label{eq:LPCC augmented}
\begin{array}{cl}
\displaystyle{
\operatornamewithlimits{\mbox{\bf minimize}}_{(x,y,z) \,
\in \, \mathbb{R}^{n+2m}; \ s \, \in \, \mathbb{R}^j_+}
} & \theta \, \triangleq \, 
\wh{c}^{\, \top} \wh{x} + e^{\top}y + f^{\top}z \\ [0.1in]
\mbox{\bf subject to} & \wh{A} \, \wh{x} + \wh{B}y + \wh{C}z 
+ \wh{D}s \, = \, b \\ [0.1in]
\mbox{\bf and} & 0 \, \leq \, y \, \perp \, z \, \geq \, 0.
\end{array} \end{equation}
We mention two important special cases of the LPCC:

\gap

\noindent {\bf Forward quadratic programs:} 
Consider a quadratic program (QP):
\begin{equation} \label{eq:forward QP}
\begin{array}{ll}
\displaystyle{
\operatornamewithlimits{\mbox{\bf minimize}}_{x \in \mathbb{R}^n}
} & c^{\top}x + \thalf \, x^{\top}Qx \\ [0.1in]
\mbox{\bf subject to} & Dx \, \geq \, d
\end{array} \end{equation}
where $Q$ is an $n \times n$ symmetric indefinite
matrix, $D \in \mathbb{R}^{m \times n}$, and
$d$ is an $m$-vector.  It is a classical result 
\cite{GiannessiTomasin73}
that if (\ref{eq:forward QP}) has a finite optimal 
solution, then the problem is equivalent to the 
LPCC, whose constraints 
are just the Karush-Kuhn-Tucker (KKT) conditions of the QP: 
\begin{equation} \label{eq:LPCC of QP}
\begin{array}{ll}
\displaystyle{
\operatornamewithlimits{\mbox{\bf minimize}}_{
(x,s,\lambda) \in \mathbb{R}^{n + 2m}}
} & \thalf \,
( c^{\top}x + d^{\top} \lambda )
\\ [0.1in]
\mbox{\bf subject to} & Qx - D^{\top} \lambda \, = \, -c 
\\ [0.1in]
& Dx - s \, = \, d \\ [0.1in]
\mbox{\bf and} & 0 \, \leq \, s \, \perp \, 
\lambda \, \geq \, 0,
\end{array} \end{equation}
where $s$ and $\lambda$ are the slack and multiplier 
vectors of 
the constraint of (\ref{eq:forward QP}),
respectively.   While there is an extension of this result 
without the solvability pre-condition, we
will focus in this paper only on the solvable QPs.  A 
noteworthy remark about the two problems
(\ref{eq:forward QP}) and (\ref{eq:LPCC of QP}) is that 
they are equivalent only in terms of their globally optimal
solutions, i.e., if $\bar{x}$ is a globally optimal 
solution of 
the former, then the pair $( \bar{x},\bar{\lambda} )$,
where $\bar{\lambda}$ is an optimal multiplier associated 
with $\bar{x}$, is globally optimal 
for (\ref{eq:LPCC of QP}),
and vice versa.  We will discuss more about the connections 
between the locally optimal solutions of these two
problems subsequently.  

\gap

\noindent {\bf Inverse affine problems:}  
In general, an inverse optimization problem
is one where the input to a (forward) optimization problem
is sought so that the modified input and the associated
optimal solution satisfy a secondary criterion.   When the 
forward problem is a convex quadratic program
(\ref{eq:forward QP}) where $Q$ is (symmetric)
positive semidefinite, and the secondary criterion is to 
minimize the deviation of the triple $(c,d,x)$, where $x$ is
an optimal solution of (\ref{eq:forward QP}), 
measured in a polyhedral norm, say $\ell_1$,
from a given triple $( \wh{c}, \wh{d}, \wh{x})$,
then this inverse (convex) QP can be formulated as the 
following LPCC:
\begin{equation} \label{eq:inverse QP}
\begin{array}{ll}
\displaystyle{
\operatornamewithlimits{\mbox{\bf minimize}}_{ (c,d,x) \in 
\mathbb{R}^{n+m+n}}
} & \| \, ( c,d,x ) - ( \wh{c}, \wh{d}, \wh{x} ) \, \|_1 
\\ [0.1in]
\mbox{\bf subject to} & Qx - D^{\top} \lambda \, = \, -c 
\\ [0.1in]
& Dx - s \, = \, d \\ [0.1in]
\mbox{\bf and} & 0 \, \leq \, s \, \perp \, \lambda \, \geq \, 0.
\end{array} \end{equation} 
In principle, the matrix $Q$ does not need to be positive semidefinite (or even symmetric) in the latter 
LPCC (\ref{eq:inverse QP});
in this case, the problem is an instance of an 
inverse LPCC and its constraints are the stationary 
conditions
of the indefinite QP (\ref{eq:forward QP}), or equivalently,
an affine variational inequality.  In an
inverse problem, the design variable $x$ is often subject to
linear constraints, leading to the augmented 
LPCC (\ref{eq:LPCC with x constraint}); see
(\ref{eq:LPCC in InvAVI}) for an instance that we employ in
the numerical experiments.

\section{The MILP formulations} 
\label{sec:MILP formulations}

Suppose that the polyhedral set
\[
P \, \triangleq \, \left\{ \, v \, \triangleq \, 
( x,y,z ) \, \in \, \mathbb{R}^n \times \mathbb{R}^{2m}_+ 
\, \mid \, Ax + By + Cz \, = \, b \, \right\}
\]
is bounded so that for a suitable scalar $\overline{M}$, 
\begin{equation} \label{eq:bound M}
\max_{1 \leq i \leq m} \, \max( \, y_i, \, z_i \, ) 
\, \leq \,
\overline{M}, \epc \forall \, ( y,z ) \mbox{ for which
$\exists \, x$ such that $( x,y,z ) \in P$}.
\end{equation}
This boundedness assumption is commonly used in 
the integer programming literature to account for bilinear
structures. In case of unbounded variables, other 
formulations that are supported by MIP solvers, such as 
reformulation with special ordered sets of type 1, or the
application of the logical Benders method for the LPCC
as described in
\cite{HuMitchellPangBennetKunapuli08}, may also be 
possible for use within the algorithm to be proposed. 
Nevertheless, the study of this advanced issue is 
beyond the scope of our present work.
It follows from (\ref{eq:bound M}) that the 
LPCC \eqref{eq:LPCC} is equivalent to the MILP:
\begin{equation} \label{eq:FIP of LPCC}
\begin{array}{ll}
\displaystyle{
\operatornamewithlimits{\mbox{\bf minimize}}_{(x,y,z,w) \in \mathbb{R}^{n + 3m}}
} & c^{\top}x + e^{\top}y + f^{\top}z \\ [0.1in]
\mbox{\bf subject to} & Ax + By + Cz \, = \, b \\ [0.1in]
\mbox{\bf and} & \left\{ \begin{array}{l}
0 \, \leq \, y_i \, \leq \, \overline{M} \, w_i \\ [0.1in] 
0  \, \leq \, z_i \, \leq \, \overline{M} \, ( 1 - w_i ) \\ [0.1in]
w_i \, \in \, \{ \, 0, \, 1 \, \}
\end{array} \right\} \ i \, = \, 1, \cdots, m,
\end{array} \end{equation}
where the complementarity condition that either $y_i=0$ or $z_i=0$ is captured by introducing a binary variable $w_i$ for each $i$.
We call \eqref{eq:FIP of LPCC} the {\sl full MILP} formulation of the LPCC (\ref{eq:LPCC}).  The computational workhorse of the
PIP method to be presented later is a reduced-MILP that is defined relative to a feasible iterate 
$\mathbf{\bar{v}} \triangleq ( \bar{x},\bar{y},\bar{z} )$ 
of (\ref{eq:LPCC}). 
Admittedly, such a feasible iterate may not be readily available in general; 
nevertheless for the QP-derived LPCC (\ref{eq:LPCC of QP}), a feasible solution of the LPCC is just
a stationary solution of the QP (\ref{eq:forward QP}) which
(if exists) can be obtained readily by a host of standard QP
methods. A similar remark applies to (\ref{eq:inverse QP}).   Throughout this paper, we
assume that we have on hand a feasible $\mathbf{\bar{v}}$ to (\ref{eq:LPCC}).
Given $\mathbf{\bar{v}}$, we let ${\cal M}_c(\mathbf{\bar{v}})$, ${\cal M}_y^+(\mathbf{\bar{v}})$, and
${\cal M}_z^+(\mathbf{\bar{v}})$ be three index sets partitioning $\{ 1, \cdots, m \}$ such that
\begin{equation} \label{eq:index sets condition}
{\cal M}_y^+(\mathbf{\bar{v}}) \, \subseteq \, \{ i \, \mid \bar{y}_i > 0 \} \ \mbox{ and } \
{\cal M}_z^+(\mathbf{\bar{v}}) \, \subseteq \, \{ i \, \mid \bar{z}_i > 0 \},
\end{equation}
which imply ${\cal M}_c(\mathbf{\bar{v}}) \supseteq 
\{ i \, \mid \bar{y}_i = \bar{z}_i = 0 \}$.  Define the 
reduced-LPCC:
\begin{equation} \label{eq:partial LPCC}
\mbox{reduced-LPCC}(\mathbf{\bar{v}}) 
\, : \, \left\{ \begin{array}{ll}
\displaystyle{
\operatornamewithlimits{\mbox{\bf minimize}}_{(x,y,z) \in \mathbb{R}^{n + 2m}}
} & c^{\top}x + e^{\top}y + f^{\top}z \\ [0.1in]
\mbox{\bf subject to} & Ax + By + Cz \, = \, b \\ [0.1in]
& 0 \, \leq \, y_i \, \perp \, z_i \, \geq \, 0, \epc 
i \, \in \, {\cal M}_c(\mathbf{\bar{v}}) \\ [0.1in]
& 0 \, = \, y_i, \ z_i \, \geq \, 0, \hspace{0.3in}  
i \, \in \, {\cal M}_z^+(\mathbf{\bar{v}})  \\ [0.1in]
\mbox{ \bf and } & 0 \, = \, z_i, \ 
y_i \, \geq \, 0, \hspace{0.3in}  i \, \in \, 
{\cal M}_y^+(\mathbf{\bar{v}}),
\end{array} \right.
\end{equation} 
which has a {\sl reduced-MILP formulation}:
\begin{equation} \label{eq:partial MILP}
\mbox{reduced-MILP}(\mathbf{\bar{v}}) 
\, : \, \left\{ \begin{array}{ll}
\displaystyle{
\operatornamewithlimits{\mbox{\bf minimize}}_{(x,y,z,w) 
\in \mathbb{R}^{n + 3m}}
} & c^{\top}x + e^{\top}y + f^{\top}z \\ [0.1in]
\mbox{\bf subject to} & Ax + By + Cz \, = \, b \\ [0.1in]
& \left\{ \begin{array}{l}
0 \, \leq \, y_i \, \leq \, \overline{M} \, w_i \\ [0.1in] 
0  \, \leq \, z_i \, \leq \, \overline{M} \, ( 1 - w_i ) 
\\ [0.1in]
w_i \, \in \, \{ \, 0, \, 1 \, \}
\end{array} \right\} \ i \, \in \, {\cal M}_c(\mathbf{\bar{v}}) 
\\ [0.4in]
& 0 \, = \, y_i, \ z_i \, \geq \, 0, \epc  
i \, \in \, {\cal M}_z^+(\mathbf{\bar{v}})  \\ [0.1in]
\mbox{ \bf and } & 0 \, = \, z_i, \ y_i \, \geq \, 0, 
\epc i \, \in \, {\cal M}_y^+(\mathbf{\bar{v}}).
\end{array} \right.
\end{equation}
Since these two reduced problems at the core of 
the computational method (to be proposed) for solving 
(\ref{eq:LPCC}), it would be of interest to understand
their role relative to the original LPCC (\ref{eq:LPCC}).  
This role is presented
in the following result which characterizes
the local optimality of a feasible triple $\mathbf{\bar{v}}$ 
to the LPCC (\ref{eq:LPCC}) in terms of a fixed-point
property with respect to the 
associated reduced-LPCC 
(\ref{eq:partial LPCC}) at $\mathbf{\bar{v}}$.

\begin{proposition} \label{pr:partial MILP and full} \rm
Suppose that $\overline{M}$ is such that (\ref{eq:bound M}) 
holds.  Let the triple 
$\mathbf{\bar{v}} \triangleq ( \bar{x},\bar{y},\bar{z} )$ 
be feasible to (\ref{eq:LPCC}).  
Let ${\cal M}_y^+(\mathbf{\bar{v}})$ and 
${\cal M}_z^+(\mathbf{\bar{v}})$ be an arbitrary
pair of index sets satisfying (\ref{eq:index sets condition}) 
and let ${\cal M}_c(\mathbf{\bar{v}})$ be the complement
of ${\cal M}_y^+(\mathbf{\bar{v}}) \cup {\cal M}_z^+
(\mathbf{\bar{v}})$ in $\{ 1, \cdots, m \}$.
Then $\mathbf{\bar{v}}$ is locally optimal for the 
reduced-LPCC$(\mathbf{\bar{v}})$
if and only if $\mathbf{\bar{v}}$ is locally optimal for 
the LPCC (\ref{eq:LPCC}).
\end{proposition}

\begin{proof}  Since the feasible set of 
(\ref{eq:partial LPCC}) is obviously a subset of
the feasible set of (\ref{eq:LPCC}), and
$\mathbf{\bar{v}}$ is also feasible to (\ref{eq:partial LPCC}),
the ``if'' statement is obvious.  Conversely, suppose 
$\mathbf{\bar{v}}$ is locally optimal for the
reduced-LPCC $(\mathbf{\bar{v}})$.  
To show that $\mathbf{\bar{v}}$ is locally optimal for the 
LPCC (\ref{eq:LPCC}), let ${\cal N}$ be a neighborhood of 
$\mathbf{\bar{v}}$ within which 
$\mathbf{\bar{v}}$ is optimal for (\ref{eq:partial LPCC}).
Let
\[
\bar{\cal N} \, \triangleq \, {\cal N} \, \displaystyle{
\bigcap_{i \in {\cal M}_y^+(\bar{\boldsymbol{v}})}
} \, \left\{ \, \boldsymbol{v} \, \in \, 
\mathbb{R}^{n+2m} \, \left| \
| \, y_i - \bar{y}_i \, | \, < \, \displaystyle{
\frac{\bar{y}_i}{2}
} \, \right. \right\} \, \displaystyle{
\bigcap_{i \in {\cal M}_z^+(\bar{\boldsymbol{v}})}
} \, \left\{ \, \boldsymbol{v} \, \in \, \mathbb{R}^{n+2m} 
\, \left| \ | \, z_i - \bar{z}_i \, | \, < \, \displaystyle{
\frac{\bar{z}_i}{2}
} \, \right. \right\}.
\]

\noindent It then follows that for all $v = (x,y,z) \in \bar{\cal N}$
with $(y,z) \geq 0$, we have $y_i > 0$ for all 
$i \in {\cal M}_y^+(\mathbf{\bar{v}})$
and $z_i > 0$ for all 
$i \in {\cal M}_z^+(\mathbf{\bar{v}})$. 
Thus, any such $(x,y,z)$ that
is feasible to (\ref{eq:LPCC}) must be feasible to 
(\ref{eq:partial LPCC}). 
\end{proof}

\subsection{Suboptimal KKT solutions in the forward QP}

As previously mentioned, the forward 
QP \eqref{eq:forward QP} 
admits an equivalent LPCC reformulation in terms of globally
optimal solutions. Unfortunately, such equivalence fails 
to hold for locally optimal solutions, as illustrated by 
Example~\ref{ex:a counterex} below.  This non-equivalence 
implies that, despite 
Proposition~\ref{pr:partial MILP and full}, a locally 
optimal solution to the reduced-LPCC, 
(which is a 
suboptimal KKT solution) does not necessarily yield a
locally optimal solution to the original QP.  This raises 
two 
natural questions: (i) to what solution of the original QP 
does an optimal solution of a reduced-LPCC 
correspond; (ii) under what conditions can the local 
optimality be asserted 
as in Proposition~\ref{pr:partial MILP and full} between 
the local minimizers of the restricted-KKT problem and the 
QP~\eqref{eq:forward QP}?  We address them in sequence after
presenting the following example.

\begin{example} \label{ex:a counterex}  \rm

Consider the nonconvex quadratic program:
\begin{equation} \label{eq:illustrative QP} 
\displaystyle{
\operatornamewithlimits{\mbox{\bf minimize}}_{
x \, \in \, [ \, 0,1 \, ]}
} \ - \thalf \, ( x^2 - x \, ),
\end{equation}
which has two global minimizers: $x = 0, 1$.
The equivalent (in terms of globally optimal solutions) LPCC
formulation is:
\[ \begin{array}{ll}
\displaystyle{
\operatornamewithlimits{\mbox{\bf minimize}}_{x,\lambda,\mu}
} & \thalf \, x - \lambda \\ [0.1in]
\mbox{\bf subject to} & 0 \, = \, \thalf - x + \lambda - \mu 
\\ [0.1in]
& 0 \, \leq \, x \, \perp \, \mu \, \geq \, 0 \\ [0.1in]
\mbox{\bf and } & 0 \, \leq \, 1 - x \, \perp \, \lambda \, 
\geq \, 0.
\end{array} 
\]
The triple $(\bar{x},\bar{\lambda},\bar{\mu} ) = 
( \thalf, 0, 0 )$ satisfies the KKT conditions
and $x = 1/2$ is a global maximizer of the given 
QP (\ref{eq:illustrative QP}).  This KKT triple is also isolated
in that there is no other KKT triple near it.   This observation
illustrates the first (negative) connection between the LPCC 
and the QP; namely,

\gap

\noindent 
$\bullet $ an isolated KKT point is not necessarily 
a local minimizer of the QP.

\gap

\noindent 
Moreover, since $\mathbf{\bar{v}} \triangleq 
(\bar{x},\bar{\lambda},\bar{\mu} )$ is an
isolated KKT triple, it must be a local minimizer
of the LPCC.  This consequence yields the second (negative)
connection between the LPCC and the QP; namely,

\gap

\noindent 
$\bullet $ a local minimizer of the LPCC is not
necessarily a local minimizer of the QP.  \hfill $\Box$ 
\end{example}

\subsubsection{A perspective from polyhedral geometry}

For the QP-derived LPCC (\ref{eq:LPCC of QP}), the 
reduced-LPCC corresponding to a 
KKT triple $\mathbf{\bar{v}} \triangleq 
( \bar{x},\bar{s},\bar{\lambda} )$ is formed with the pair
of index sets ${\cal M}_s^+(\mathbf{\bar{v}})$ and 
${\cal M}_{\lambda}^+(\mathbf{\bar{v}})$ satisfying
\begin{equation} \label{eq:KKT ndex sets condition}
{\cal M}_s^+(\mathbf{\bar{v}}) \, \subseteq \, 
\{ i \, \mid \bar{s}_i > 0 \} \ \mbox{ and } \
{\cal M}_{\lambda}^+(\mathbf{\bar{v}}) \, \subseteq \, 
\{ i \, \mid \bar{\lambda}_i > 0 \},
\end{equation}
along with the complement ${\cal M}_c(\mathbf{\bar{v}})$ 
of ${\cal M}_s^+(\mathbf{\bar{v}}) \cup 
{\cal M}_{\lambda}^+(\mathbf{\bar{v}})$ in 
$\{ 1, \cdots, m \}$:
\begin{equation} \label{eq:restricted KKT}
\begin{array}{l}
\mbox{restricted-KKT}(\mathbf{\bar{v}}) \\ [3pt]
\equiv \mbox{ a partial LPCC}; \\ [5pt]
\mbox{optimal solution set} \\ [3pt]
\mbox{denoted $S(\mathbf{\bar{v}})$} 
\end{array}
\left\{ \begin{array}{ll}
\displaystyle{
\operatornamewithlimits{\mbox{\bf minimize}}_{(x,s,\lambda) 
\in \mathbb{R}^{n + 2m}}
} & c^{\top}x + d^{\top} \lambda \\ [0.1in]
\mbox{\bf subject to} & Qx - D^{\top} \lambda \, = \, -c 
\\ [0.1in]
& Dx - s \, = \, d \\ [0.1in]
& 0 \, \leq \, s_i \, \perp \, \lambda_i \, \geq \, 0, 
\epc i \, \in \, {\cal M}_c(\mathbf{\bar{v}}) \\ [0.1in]
& 0 \, = \, s_i \, \leq \, \lambda_i, \hspace{0.45in} 
i \, \in \, {\cal M}_{\lambda}^+(\mathbf{\bar{v}}) 
\\ [0.1in]
\mbox{\bf and} & 0 \, = \, \lambda_i \, \leq \, s_i, 
\hspace{0.45in} i \, \in \, {\cal M}_s^+(\mathbf{\bar{v}}),
\end{array} \right.
\end{equation}
which has an equivalent MILP formulation:
\begin{equation} \label{eq:MILP restricted KKT}
\mbox{\begin{tabular}{l}
MILP of \\ 
restricted-KKT$(\mathbf{\bar{v}})$ :
\end{tabular}} \left\{ \begin{array}{ll}
\displaystyle{
\operatornamewithlimits{\mbox{\bf minimize}}_{
(x,s,\lambda,w) \in \mathbb{R}^{n + 3m}}
} & c^{\top}x + d^{\top} \lambda \\ [0.1in]
\mbox{\bf subject to} & Qx - D^{\top} \lambda \, = \, -c 
\\ [0.1in]
& Dx - s \, = \, d \\ [0.1in]
& \left\{ \begin{array}{l}
0 \, \leq \, s_i \, \leq \, \overline{M} \, w_i \\ [0.1in]
0 \, \leq \, \lambda_i \, \leq \, \overline{M} ( 1 - w_i ) 
\\ [0.1in]
w_i \, \in \, \{ \, 0, \, 1 \, \}
\end{array} \right\}  \ \in \, {\cal M}_c(\mathbf{\bar{v}}) 
\\ [0.4in]
& 0 \, = \, s_i \, \leq \, \lambda_i, \hspace{0.45in} 
i \, \in \, {\cal M}_{\lambda}^+(\mathbf{\bar{v}}) \\ [0.1in]
\mbox{\bf and} & 0 \, = \, \lambda_i \, \leq \, s_i, 
\hspace{0.45in} i \, \in \, {\cal M}_s^+(\mathbf{\bar{v}}),
\end{array} \right.
\end{equation}
provided that the feasible set of the 
QP (\ref{eq:forward QP}) 
has a strictly feasible point that ensures the existence of 
the scalar $\overline{M}$.  

\gap

Clearly, the feasible set of the 
restricted-KKT$(\mathbf{\bar{v}})$ is a subset of the 
feasible region of \eqref{eq:LPCC of QP}; so is
the optimal set $S(\mathbf{\bar{v}})$.  For any set
$S \subseteq \mathbf{R}^{n+2m}$, let
$\proj_x(S)$ be the projection 
of $S$ onto the $x$-space.  For ease of reference, let 
$q(x)\triangleq c^\top x + \thalf x^\top Q x$ and 
$\ell(x,\lambda) \triangleq c^\top x+d^\top \lambda$ be the 
objectives of the QP~\eqref{eq:forward QP} and the 
LPCC~\eqref{eq:LPCC of QP} respectively.  Observe that for any 
feasible solution $( x,\lambda,s )$ to the 
LPCC~\eqref{eq:LPCC of QP}, one always has 
\begin{equation}\label{eq:equal obj}
\ell( x,\lambda) \, = \, c^\top x + ( Dx - s )^\top \lambda 
\, = \, c^\top x + ( Qx + c )^\top x \, = \, 2q(x).
\end{equation}
Define $P^{\, \prime} \triangleq 
\left\{ (x,s,\lambda) \in \R^{n+2m} \mid 
Qx-D^\top\lambda=-c,\,Dx-s=d,\,s\ge0,\,\lambda \ge 0 \right\}$. 
For any set $R$, the convex hull of $R$ is denoted by 
$\conv(R)$.  We recall the distinction between a 
KKT triple and a KKT point of the QP (\ref{eq:forward QP}).
The former is a triple $(x,s,\lambda)$ satisfying the KKT 
conditions of the QP, whereas the latter is a (primal 
feasible) vector $x$ for which
$\lambda$ exists such that $(x,s,\lambda)$ is a KKT triple. 
In this terminology, Proposition~\ref{pr:conv} states that 
for any subset $S$ of KKT triples, if $(x,s,\lambda)$ is an
optimal KKT triple of the linear function $\ell$ on $S$, 
then,
$x$ is an optimal solution of the quadratic function $q$ on
$\conv(\proj_x(S))$.  Two consequences of this result are 
provided,
one of which answers the first question (i) raised above, 
and the other answers partially the second question (ii).

\begin{proposition}\label{pr:conv} \rm 
Let $S$ be a subset of the 
feasible region of \eqref{eq:LPCC of QP}.
If $(x^*,s^*,\lambda^*) \in \displaystyle{
\operatornamewithlimits{\mbox{\bf argmin}}_{
(x,s,\lambda) \in S}
} \, \ell(x,\lambda)$, then
$x^* \in\displaystyle{
\operatornamewithlimits{\mbox{\bf argmin}}_{
x \in \conv( \proj_x(S) )}
} \,  \, q(x)$.  
\end{proposition}

\begin{proof}
Consider any $\wh{x} \in\conv(T)$, where
$T \triangleq \proj_x S$.  We have 
$\conv(T) = \proj_x \conv(S)$, implying that there 
exists $( \wh{s},\wh{\lambda} )$ such that 
$(\wh{x},\wh{s},\wh{\lambda} )$ belongs to $\conv(S)$.  Since
$\ell(\bullet)$ is a linear function, we can deduce 
$(x^*,s^*,\lambda^*) \in \displaystyle{
\operatornamewithlimits{\mbox{\bf argmin}}_{
(x,s,\lambda) \in \conv(S)}
} \, \ell(x,\lambda)$; thus 
$\ell(x^*,\lambda^*) \leq \ell( \wh{x},\wh{\lambda})$. 
Moreover, $S \subseteq P^{\, \prime}$ implies that 
$( \wh{x},\wh{s},\wh{\lambda} )\in P^{\, \prime}$. 
Hence we have
\[ \begin{array}{lll}
2q(x^*) & = & \ell(x^*,\lambda^*) 
\, \leq \, \ell( \wh{x},\wh{\lambda} ) \\ [0.1in]
& = & c^\top \wh{x} + d^\top \wh{\lambda} \\ [0.1in]
& = & c^\top \wh{x} + d^\top \wh{\lambda} + 
\wh{x}^\top \left(c + Q \wh{x} - D^\top \wh{\lambda} \right) 
\\ [0.1in]
& = & 2c^\top \wh{x} + \wh{x}^\top Q \wh{x} - 
\wh{\lambda}^\top ( D \wh{x} - d ) \\ [0.1in]
& \leq & 2c^\top \wh{x} + \wh{x}^\top Q\wh{x} \, = \, 
2q( \wh{x} ),
\end{array} \]
establishing the desired conclusion.
\end{proof}

The above proposition yields the following corollary.

\begin{corollary} \label{co:two consequences} \rm
The following two statements hold for
a triple $\mathbf{v}^* \triangleq (x^*,s^*,\lambda^*)$.

\gap

\noindent {\bf (a)} If $\mathbf{v}^*$ is an optimal 
solution of the 
restricted KKT$(\mathbf{v^*})$, then $x^*$ is a minimizer
of the quadratic function $q(x)$ on 
$\conv(\proj_x(S(\mathbf{v^*})))$.

\gap

\noindent {\bf (b)} If $\mathbf{v}^*$ is a local minimizer
of the qp-induced LPCC (\ref{eq:LPCC of QP}), then 
a neighborhood ${\cal N}$ of $\mathbf{v}^*$ exists 
such that  
$x^* \in \displaystyle{
\operatornamewithlimits{\mbox{\bf argmin}}_{
x \in \conv( \proj_x( \bar{S} \, \cap \, {\cal N} ) )}
} \, q(x)$, where $\bar{S}$ is the set of KKT triples of the
QP~(\ref{eq:forward QP}).
\end{corollary}

\begin{proof}  
For the proof of statement (a), it suffices to take
$S = S(\mathbf{v^*})$ in Proposition~\ref{pr:conv}.  
For the proof of (b), it
suffices to take ${\cal N}$ to be a neighborhood of
$\mathbf{v^*}$ within which $\mathbf{v^*}$ is a minimizer of
$\ell$ on $\bar{S}$ and take $S = \bar{S} \cap {\cal N}$
in the same proposition.
\end{proof}

Proposition~\ref{pr:conv} is clear if $Q$ is negative 
semidefinite because there always exists a minimizer at 
an extreme point of the feasible region for a concave 
program, provided that such a problem has an optimal
solution.  On the other hand, if $q(\bullet)$ is a 
strictly convex 
function, then the set of KKT triples is a singleton; hence
Proposition~\ref{pr:conv} is not 
informative in this case.  This remark leads to a refinement
of the above proposition.

Let $L$ be any linear subspace of $\R^n$ such that 
$x^\top Q x\ge 0$ for all $x\in L$, i.e.\ $q(\bullet)$ is 
convex over $ L$.  For instance, because $Q$ admits an 
orthonormal decomposition $Q=Q^+-Q^-$, where $Q^+$ and 
$Q^-$ are positive semidefinite, one can take $L$ to be 
any linear subspace formed from 
$\text{Range}(Q^{+})$. 
The result below (Proposition~\ref{pr:conv-ext})
easily recovers the well-known result
(Corollary~\ref{co:psd KKY suff}) that the KKT conditions
are sufficient for global optimality for a convex 
(quadratic) program.  This corollary is stated only to 
highlight that the
proposition is a generalization of this classical result.

\begin{proposition} \label{pr:conv-ext} \rm
Let $S$ be any {\sl convex} subset of the 
feasible region of \eqref{eq:LPCC of QP}.
For any triple $(x^*,s^*,\lambda^*) \in \displaystyle{
\operatornamewithlimits{\mbox{\bf argmin}}_{
(x,s,\lambda) \in S}
} \, \ell(x,\lambda)$, it holds that
$x^* \in \displaystyle{
\operatornamewithlimits{\mbox{\bf argmin}}_{
x \, \in \, T_L \, \cap \, P}
} \, q(x)$, for any linear subspace $L$ on which $Q$
is positive semidefinite, where 
$T_L \triangleq \proj_x S + L$.
\end{proposition}

\begin{proof}  Let $x \in T_L \cap P$ be arbitrary and let
$y \in \proj_x S$ be such that $y - x \in L$.  Since $S$
is convex, so is $\proj_x(S)$.  Thus $q(y) \geq q(x^*)$
by Proposition~\ref{pr:conv}.  Moreover, elements of
$\proj_x S$ are KKT points; hence 
$\nabla q(y)^{\top}( x - y ) \geq 0$.  Consequently,
\[ \begin{array}{lll}
q(x) - q(x^*) & = & [ \, q(x) - q(y) \, ] + 
[ \, q(y) - q(x^*) \, ] \\ [0.1in]
& \geq & \nabla q(y)^{\top} ( \, x - y \, ) + \thalf \, 
( \, x - y \, )^{\top} Q( \, x - y \,) \, \geq \, 0
\end{array} \]
by the convexity of $q(\bullet)$ on the subspace $L$.
\end{proof}

\begin{corollary} \label{co:psd KKY suff} \rm
If $Q$ is positive semidefinite, then every KKT point is a 
global minimizer of the QP (\ref{eq:forward QP}).
\end{corollary}

\begin{proof} It suffices to let the subset $S$ in 
Proposition~\ref{pr:conv-ext} be a single KKT triple and
$L = \mathbb{R}^n$. 
\end{proof}

\subsubsection{Aiming for local optimality for the forward QP} 
\label{subsec:local sol QP}

The discussion below aims to obtain a relation
between the partial KKT$(\mathbf{\bar{v}})$ and the QP
(\ref{eq:forward QP}) analogous to 
Proposition~\ref{pr:partial MILP and full}, 
by assuming that $\mathbf{\bar{v}}$
has a similar fixed-point property as in the proposition.
A full answer to this extended issue turns out to be not easy.   
In what follows, we provide a partial answer.  
Before undertaking the analysis, it would be useful to 
explain the motivation.  Indeed, it is long known that
the second-order necessary condition of optimality is also 
sufficient for a KKT vector to be a local minimizer
of a QP; see \cite[Chapter~3]{LeeTamYen05} for a comprehensive 
theory of quadratic programming.  Nevertheless,
this condition is in terms of a matrix-theoretic copositivity 
condition.  The analysis below aims to provide a sufficient
condition under which a KKT vector can be shown to be a local 
minimizer of the QP, by way of a relaxation of
the partial LPCC formulation of the KKT conditions in the 
spirit of (\ref{eq:partial LPCC}).  Since the solution 
of the resulting LPCC can be accomplished in the same way as 
(\ref{eq:partial LPCC}), which is the computational
workhorse of the PIP method to be described later; as such, the 
resulting LPCC condition is practically more viable
than the copositivity condition, albeit the former is only 
sufficient whereas the latter is necessary and sufficient.

\gap
 
Toward the goal stated above, we first present a simple 
observation whose converse provides the unresolved key 
to this issue.

\begin{proposition} \label{pr:locmin of QP implies} \rm
Let $\bar{x}$ be a locally optimal solution 
of (\ref{eq:forward QP}) and let $\bar{\lambda}$ be an
arbitrary optimal multiplier associated with $\bar{x}$.  
It then follows that the triple 
$\mathbf{\bar{v}} \triangleq 
( \bar{x},\bar{s},\bar{\lambda} )$, where 
$\bar{s} \triangleq D\bar{x} - d$,
is a locally optimal solution of the LPCC (\ref{eq:LPCC of QP}).
\end{proposition}

\begin{proof} Let ${\cal N}$ be a neighborhood of $\bar{x}$ 
within which $\bar{x}$ is optimal for (\ref{eq:forward QP}).  
For any $x \in {\cal N}$ such that the triple
$( x,s,\lambda )$, where $s \triangleq Dx - d$, is feasible to 
(\ref{eq:LPCC of QP}), by noting \eqref{eq:equal obj}, we have
\[
c^{\top}x + d^{\top} \lambda  \, = \, 2 \, ( \, c^{\top}x + 
\thalf x^{\top}Qx \, ) \, \geq \, 
2 \, ( \, c^{\top}\bar{x} + \thalf \bar{x}^{\top}Q\bar{x} \, )
\, = \, c^{\top} \bar{x} + d^{\top} \bar{\lambda},
\]
proving the desired (local) optimality of $\mathbf{\bar{v}}$
for (\ref{eq:LPCC of QP}). 
\end{proof}

In the spirit of Proposition~\ref{pr:locmin of QP implies},
we assume that the triple $\mathbf{\bar{v}}$ is a globally 
optimal solution of the restricted-KKT$(\mathbf{\bar{v}})$.
Toward the goal of showing that $\bar{x}$ is a local 
minimizer of the QP (\ref{eq:forward QP}), we rely on
the following relaxation of (\ref{eq:restricted KKT}) 
pertaining to the multipliers $\lambda_i$:
\begin{equation} \label{eq:restricted KKT relaxed II}
\mbox{\begin{tabular}{l}
relaxation of \\
restricted KKT$(\mathbf{\bar{v}})$:
\end{tabular}} \left\{ \begin{array}{ll}
\displaystyle{
\operatornamewithlimits{\mbox{\bf minimize}}_{
(x,s,\lambda) \in \mathbb{R}^{n + 2m}}
} & c^{\top}x + d^{\top} \lambda \\ [0.1in]
\mbox{\bf subject to} & Qx - D^{\top} \lambda \, = \, -c 
\\ [0.1in]
& Dx - s \, = \, d \\ [0.1in]
& 0 \, \leq \, s_i \, \perp \, \lambda_i \, \geq \, 0, 
\epc i \, \in \, {\cal M}_c(\mathbf{\bar{v}}) \\ [0.1in]
& 0 \, = \, s_i; \ \lambda_i \mbox{ free}, \hspace{0.3in} 
i \, \in \, {\cal M}_{\lambda}^+(\mathbf{\bar{v}}) \\ [0.1in]
\mbox{\bf and} & ( \, \lambda_i, \, s_i \, ) \, \geq \, 0, 
\hspace{0.5in} i \, \in \, {\cal M}_s^+(\mathbf{\bar{v}}).
\end{array} \right.
\end{equation}
We have the following result whose assumption is
stated in terms of the global solution property of the triple 
$\mathbf{\bar{v}}$ on hand.

\begin{proposition} \label{pr:sufficient for QP local min} \rm
Suppose that the QP (\ref{eq:forward QP}) has an optimal 
solution.  Let a feasible triple 
$\mathbf{\bar{v}} \triangleq ( \bar{x},\bar{s},\bar{\lambda} )$
of (\ref{eq:LPCC of QP}) be optimal for 
(\ref{eq:restricted KKT relaxed II}).  Then $\bar{x}$ is 
a locally optimal solution for the QP (\ref{eq:forward QP}).
\end{proposition}

\begin{proof}  The proof is divided into two steps.  First, 
we show that $\bar{x}$ is a globally optimal solution of
the restricted QP:
\begin{equation} \label{eq:restricted QP}
\begin{array}{ll}
\displaystyle{
\operatornamewithlimits{\mbox{\bf minimize}}_{
x \in \mathbb{R}^n}
} & c^{\top}x + \thalf \, x^{\top}Qx \\ [0.1in]
\mbox{\bf subject to} & ( Dx \, \geq \, d )_j, \epc 
j \, \not\in \, {\cal M}_{\lambda}^+(\mathbf{\bar{v}}) 
\\ [0.1in]
\mbox{\bf and} & ( Dx \, = \, d )_j, \epc 
j \, \in \, {\cal M}_{\lambda}^+(\mathbf{\bar{v}}),
\end{array} \end{equation}
which must have an optimal solution $\wh{x}$ because 
its feasible set is a subset of that of the solvable 
QP (\ref{eq:forward QP}).
Together with a multiplier $\wh{\lambda}$ and slack 
$\wh{s} \triangleq D\wh{x} - d$, the triple 
$( \wh{x},\wh{s},\wh{\lambda} )$
is feasible to (\ref{eq:restricted KKT relaxed II}),  Hence 
we have 
\[
c^{\top} \wh{x} + \thalf \, \wh{x}^{\, \top}Q \wh{x} 
\, = \, 
\thalf [ \, c^{\top} \wh{x} + d^{\top} \wh{\lambda} \, ]
\, \geq \, 
\thalf [ \, c^{\top} \bar{x} + d^{\top} \bar{\lambda} \, ]
\, = \, c^{\top} \bar{x} + 
\thalf \, \bar{x}^{\top}Q \bar{x}, 
\]
where the last equality comes from \eqref{eq:equal obj}.
The next step is to show that the second-order necessary
condition of the restricted QP (\ref{eq:restricted QP}) is
the same as that of the original (\ref{eq:forward QP}) 
at $\bar{x}$; in turn, it suffices to show that these two 
problems have the same critical cones at $\bar{x}$; 
that is, 
if ${\cal F}_{\rm qp}$ and ${\cal F}_{\rm rst}$ denote the 
two feasible sets respectively, then
\[
{\cal T}(\bar{x};{\cal F}_{\rm qp}) \, \cap \, 
\nabla q(\bar{x})^{\perp} \, = \,
{\cal T}(\bar{x};{\cal F}_{\rm rst}) \, \cap \, 
\nabla q(\bar{x})^{\perp},
\]
where ${\cal T}(\bar{x};{\cal F})$ is the tangent cone of 
a set ${\cal F}$ at $\bar{x} \in {\cal F}$.  
Since the right-hand cone is obviously a subcone of the 
left-hand cone, it suffices to show one inclusion:
\[
{\cal T}(\bar{x};{\cal F}_{\rm qp}) \, \cap \, 
\nabla q(\bar{x})^{\perp} \, \subseteq \,
{\cal T}(\bar{x};{\cal F}_{\rm rst}) \, \cap \, 
\nabla q(\bar{x})^{\perp}.
\]
Let $v$ belongs to the left-hand cone.  We then have
$D_{j \bullet} v \geq 0$ for all $j \in {\cal A}(\bar{x})$, 
which is the index set of active constraints at $\bar{x}$, 
i.e., such that $D_{j \bullet}\bar{x} = d_j$.   To show 
that $v$ belongs to the right-hand cone, it suffices to 
show that $D_{j \bullet} v = 0$ for all 
$j \in {\cal M}_{\lambda}^+(\mathbf{\bar{v}})$.  We have
\[
\nabla q(\bar{x}) \, = \, c + Q\bar{x} \, = \, 
\displaystyle{
\sum_{j \in {\cal A}(\bar{x})}
} \, \bar{\lambda}_j \, ( \, D_{j \bullet} \, )^{\top}.
\]
Thus $v^{\top}\nabla q(\bar{x}) = 0$ implies that  
$D_{j \bullet}^{\top}v = 0$ if $\bar{\lambda}_j > 0$.
Since all indices $j \in {\cal M}_{\lambda}^+
(\mathbf{\bar{v}})$ satisfies that latter condition, 
the equality of
the two critical cones therefore follows readily.
\end{proof} 

\section{The PIP Method for the LPCC} \label{sec:PIP}

The PIP method solves the problem (\ref{eq:LPCC}) under 
two basic
assumptions: (1) the problem has an optimal solution, and 
(2) the condition (\ref{eq:bound M}) holds for some scalar 
$\overline{M}$.  It further assumes that a feasible
triple $\mathbf{\bar{v}}$ of (\ref{eq:LPCC}) is available 
on hand.  The general idea of PIP is very simple and its 
steps are quite flexible, which we summarize and present 
a pseudo code for further details.  The main strategy is 
the control of the three index sets: 
${\cal M}_c(\mathbf{\bar{v}})$,
${\cal M}_y^+(\mathbf{\bar{v}})$ and 
${\cal M}_z^+(\mathbf{\bar{v}})$ in the 
reduced-LPCC$(\mathbf{\bar{v}})$
for a given triple $\mathbf{\bar{v}}$.  
This is done by   
fixing a scalar $p_{\max} \in ( 0,1 )$ so that 
$p_{\min} \triangleq 1-p_{\textrm{max}}$ 
lower-bounds the percentage of positive elements in 
$(\bar{y}, \bar{z})$ at each iteration.  Specifically,
an initial scalar $p \in ( p_{\min}, \, 1 )$ is used 
(e.g.\ $p = 0.8$ initially) and updated to control the size 
of ${\cal M}_y^+(\mathbf{\bar{v}})$ and 
${\cal M}_z^+(\mathbf{\bar{v}})$, so that in addition to satisfying (\ref{eq:index sets condition}),
\begin{equation} \label{eq:index sets in algorithm}
| \, {\cal M}_y^+(\mathbf{\bar{v}}) \, | \, = \, 
\lfloor \, p \, \times \mid 
\{ \, i \, : \, \bar{y}_i > 0 \, \} \, | \, \rfloor 
\ \mbox{ and} \
| \, {\cal M}_z^+(\mathbf{\bar{v}}) \, | \, = \,
\lfloor \, p \, \times \mid 
\{ \, i \, : \, \bar{z}_i > 0 \, \}\, | \, \rfloor,
\end{equation}
where $\lfloor \, \bullet \, \rfloor$ denotes the floor 
function. 

The percentage $p$ is kept unchanged in the next iteration 
if the minimum objective value of the 
reduced-LPCC$(\mathbf{\bar{v}})$, obtained 
by solving the
reduced-MILP$(\mathbf{\bar{v}})$, is less
than the value at the current $\mathbf{\bar{v}}$; this is
the {\bf improvement case}.  Otherwise $p$ is decreased by 
a fixed scalar $\alpha \in ( 0, \, 1)$;
thus the number of integer variables to the determined at 
the next iteration increases; this is {\bf no improvement
case}. Furthermore, we fix a priori a positive integer 
$r_{\max}$ (say 3 or 4) to be 
the maximum number of iterations $p$ is kept unchanged; when
this number of times is reached, $p$
is decreased (by $\alpha$) in the next iteration.
The algorithm terminates if the maximum percentage 
$p_{\max}$ of un-fixed integer variables is reached and 
there is no improvement in the 
minimum objective value of  (\ref{eq:partial LPCC}) or 
if $p$ is unchanged and equal to $p_{\min}$ for $r_{\max}$ 
iterations. 

\begin{algorithm}
\caption{The PIP$(p_{\max})$ method for LPCC}\label{alg:PIP}
\begin{algorithmic}[1] 
\State{\bf Initialization:} Let $p_{\max} \in (0, 1)$, 
$r_{\max}$ be a positive integer, 
and $\alpha \in ( 0, 1 )$ be fixed.  
\State Let an initial feasible triple 
$\mathbf{\bar{v}} \triangleq ( \bar{x},\bar{y},\bar{z} )$ 
of (\ref{eq:LPCC}) be given.  Let 
$\bar{\theta} \triangleq 
c^{\top}\bar{x} + e^{\top}\bar{y} + f^{\top} \bar{z}$.  
\State Let $p_{\min} \triangleq 1 - p_{\max}$.
Let $p \in ( p_{\min}, \, 1 )$ be an initial 
percentage.  Let $r = 1$ initially.
\While{the stopping criterion ($p < p_{\min}$) is not met}
\State  Let the pair ${\cal M}_y^+(\mathbf{\bar{v}})$ and 
${\cal M}_z^+(\mathbf{\bar{v}})$ satisfy
(\ref{eq:index sets condition}) and 
(\ref{eq:index sets in algorithm}) and let 
${\cal M}_c(\mathbf{\bar{v}})$
be the complement of these two index sets.  
\State Solve the 
reduced-MILP$(\mathbf{\bar{v}})$ and 
let $\mathbf{\wh{x}} \triangleq ( \wh{x},\wh{y},\wh{z},\wh{w} )$
be an optimal solution.  
 \State Set $\mathbf{\bar{v}} \leftarrow 
 ( \wh{x},\wh{y},\wh{z} )$.
\State Let $\wh{\theta} \triangleq 
c^{\top}\wh{x} + e^{\top}\wh{y} + f^{\top} \wh{z}$.
\If{$\wh{\theta} = \bar{\theta}$ \textbf{OR}  
$r \ge  r_{\max}$ }
\State $r \leftarrow 1 $
\State $p\leftarrow p - \alpha$ 
\Else
            \State $r \leftarrow r+1$
        \EndIf
 
    \EndWhile
    \State \textbf{return} $\mathbf{\bar{v}}$. 
\end{algorithmic}
\end{algorithm}

\begin{proposition} \label{pr:termination of PIP} \rm
The PIP algorithm terminates in finitely many steps. 
At termination, the output $\mathbf{\bar{v}}$ must be a 
local minimum of the LPCC (\ref{eq:LPCC}).  
\end{proposition}

\begin{proof}
At each iteration, if no improvement is observed, the percentage $p$ is decreased, leading to a different reduced subproblem being solved in the next iteration. Since both the objective value and $p$ decrease monotonically throughout the algorithm, no reduced subproblem can be solved more than twice. Given that the number of reduced subproblems is finite, the algorithm must terminate in a finite number of steps.
At the termination of the algorithm, 
the iterate $\mathbf{\bar{v}}$ must be an optimal 
solution of the 
reduced-LPCC$(\mathbf{\bar{v}})$, 
hence a local minimizer of the LPCC (\ref{eq:LPCC})
by Proposition~\ref{pr:partial MILP and full}.  
\end{proof}

When the above PIP algorithm is applied to the QP-derived 
LPCC (\ref{eq:LPCC of QP}), the subproblems are the
restricted KKT problems (\ref{eq:restricted KKT}). 
Let $S^k$ be the feasible region of the $k$-th 
restricted QP-derived LPCC  and 
$\mathbf{v}^k = (x^k,s^k,\lambda^k)$ be the value 
of $\wh{\mathbf{v}}$ at iteration $k$ in the PIP 
algorithm. Then 
$\mathbf{v}^k \in \argmin\limits_{(x,s,\lambda)\in \mathbf{S}^k}
c^{\top} x + d^{\top}\lambda$, where 
$\mathbf{S}^k \triangleq \displaystyle{
\bigcup_{i=1}^k
} \, S^i$. It follows from Proposition~\ref{pr:conv} that 
$x^k$ is the optimal solution of the QP restricted to 
$\conv( \proj_x(\mathbf{S}^k) )$.  

\gap

\noindent {\bf Example~\ref{ex:a counterex} continued.}
With
the same KKT triple $\mathbf{\bar{v}} = ( \thalf,0,0 )$,
consider the restricted-KKT$(\mathbf{\bar{v}})$ 
given by (\ref{eq:restricted KKT}) by fixing $\lambda = 0$:
\[ \begin{array}{ll}
\displaystyle{
\operatornamewithlimits{\mbox{\bf minimize}}_{x,\lambda,\mu}
} & \thalf \, x - \lambda \\ [0.1in]
\mbox{\bf subject to} & 0 \, = \, \thalf - x + \lambda - \mu 
\\ [0.1in]
& 0 \, \leq \, x \, \perp \, \mu \, \geq \, 0 \\ [0.1in]
\mbox{\bf and } & 0 \, \leq \, 1 - x, \ \lambda \, = \, 0.
\end{array} 
\]
The global minimizer of the latter LPCC is 
$\left( \, 0,0,\thalf \, \right)$,
recovering the global minimizer $x = 0$ of the QP.  Similarly,
by fixing $\mu = 0$, we obtain the other restricted 
KKT$(\mathbf{\bar{v}})$:
\[ \begin{array}{ll}
\displaystyle{
\operatornamewithlimits{\mbox{\bf minimize}}_{x,\lambda,\mu}
} & \thalf \, x - \lambda \\ [0.1in]
\mbox{\bf subject to} & 0 \, = \, \thalf - x + \lambda - \mu 
\\ [0.1in]
& 0 \, \leq \, x, \ \mu \, = \, 0 \\ [0.1in]
\mbox{\bf and } & 0 \, \leq \, 1 - x \, \perp \,  
\lambda \, \geq \, 0,
\end{array} 
\]
whose global minimum is $\left( 1,\thalf,0 \, \right)$, which
yields the other global minimizer $x = 1$ of the original QP.
Since these two restricted LPCCs are candidate subproblems
to be selected by the PIP algorithm when initiated at
the KKT triple $\mathbf{\bar{v}}$, we deduce the following
positive effect of the algorithm:

\gap

\noindent $\bullet $ starting at a KKT solution that is
not a local minimum of a QP, the PIP algorithm is capable
of computing a local minimum of the QP.  \hfill $\Box$

\section{Computational Experiments}

This section presents the numerical results of a carefully 
conducted set of computational experiments to evaluate the 
performance of PIP.  To be described in 
Subsection \ref{sec: instances},
several classes of problems are employed in the tests;
these are (a) standard quadratic program, abbreviated as
StQP, (b) quadratic assignment problem, abbreviated as
QAP, (c) inverse quadratic problem, abbreviated as InQP, 
and (d) inverse affine variational inequality, 
abbreviated as InvAVI.  
We have not tested a {\sl general} LPCC 
because we feel that it would be more meaningful to apply
the PIP method
to some structured instances of the LPCC with potential
applications.  We begin in the next subsection 
with some implementation details of 
Algorithm~\ref{alg:PIP}, this is followed by the 
description of the problem instances, and then the results. 
Our experiments show that (1) PIP can improve
considerably the solutions of QPs and LPCCs obtained by an 
off-the-shelf nonlinear programming solver (see below); 
and (2) PIP can improve the solution of the full 
MILP formulation (\ref{eq:FIP of LPCC}) defined by a complete 
set of integer variables, with and without warm start 
(which we denote by FMIP-W and FMIP, 
respectively) in terms of the computation time and solution
quality. The latter improvement is confirmed when the solution
of the full MILP (a) requires excessive time, and (b) is
terminated in roughly the same amount of time required by PIP
with solution quality often inferior to that by PIP itself.

\subsection{Implementation details}
\label{sec: implementation}

Some details in the implementation PIP($p_{\max}$) method 
(Algorithm \ref{alg:PIP}) are as follows.

\gap

\noindent {\bf Parameters:} We set the 
maximum proportion of integer variables to be 
$p_{\max} \in (0, 1)$ with the initial proportion to 
be $0.8$, the maximum number of iterations that 
$p$ is allowed to remain unchanged to be  $r_{\max} = 3$, 
and the decrease for $p$ to be $\alpha = 0.1$. 

\gap

\noindent {\bf NLP solvers for initial solutions:}  
The nonlinear programming solver to 
compute the initial solutions of PIP (and FMIP-W) can 
be changed freely based on the users' preference.  In 
fact, one of the main goals of PIP is to improve the 
suboptimal QP or LPCC solutions 
obtained through such a local search procedure. 
While it is observed in our experiments that PIP can 
improve from any feasible solution that is not globally 
optimal and has shown little sensitivity to the quality 
of the initial solutions, an 
initial solution with competitive quality is indeed usually 
preferred in practice. Hence, we first conduct a preliminary
comparison of three popular NLP solvers: 
IPOPT \cite{ipopt}, filterSQP 
\cite{filtersqp}, and Knitro \cite{knitro} on the 
NEOS server 
\cite{neos1, neos2, neos3} in solution quality for a set 
of random StQP instances of varying problem 
size (200, 500, 1000) 
and density (0.5, 0.75) (see explanation in 
Subsection~\ref{sec: instances});
we observe that the solution quality of these three 
solvers is quite similar. 
Hence, we use the free open-source IPOPT (version 3.12.13)
\cite{ipoptcode} as our NLP solver 
to obtain an initial solution for the tested problems.

\gap

A numerical issue may arise when the NLP solver fails to produce an LPCC feasible solution for PIP initialization within a reasonable time, which is the case for large InvQP and QAP instances. In such a scenario, there are three alternative approaches to start PIP. The first approach is to obtain the initial LPCC feasible solution via
other feasible solution methods, such as an (early terminated) incumbent FMIP solution obtained from MILP solver, which forces the PIP solution path to align with FMIP in the initialization phase. This is the approach we adopt in our experiments for large inverse QP instances, where PIP runs are initialized by feasible solutions obtained from solving the full MILP reformulation of LPCCs using the 
GUROBI solver under a specified time limit (10 minutes). The second approach is to warm start the NLP solver with a known feasible solution that can be constructed with domain knowledge and is particularly applicable to specially structured problems. The third approach 
 is to abandon the infeasible initial solution for the first warm start but still use it to construct the first reduced MILP by partially fixing binary variables in the same spirit of \eqref{eq:index sets in algorithm}.
The advantage of the third approach is four-fold: (1) No need for extra pipeline or domain knowledge; (2) The LPCC feasibility will be regained automatically right after solving the first sub-problem; (3) Even without warm start, the first reduced-MILP does not demand high computational power since most integer variables are fixed in early stage of PIP; (4) The information contained in the initial solution (extracted by NLP techniques) is not wasted, as it is utilized to construct the first sub-problem. 
We adopt the third approach to initialize PIP on QAP instances and use FMIP as benchmarks rather than FMIP-W.

\gap

\noindent {\bf The MILP solver:}  In our
experiments, we use
GUROBI (version 11.0.1) \cite{gurobi} to solve all the
integer programs: the full MIP and the reduced ones.
This is one of the most advanced, publicly 
available general integer programming 
solvers.  Default values in the solver are used.  Most
importantly, in the PIP iterations, the ``warmstart = true'' 
option in employed in GUROBI so that
a current iterate $\mathbf{\bar{v}}$ can be profitably 
employed to initialize the solution of the next reduced-MILP.

\gap

\noindent {\bf Construction and update of index sets:} 
With an iterate 
$\mathbf{\bar{v}}=(\bar{x}, \bar{y}, \bar{z})$ available, 
the index sets ${\cal M}_y^+(\mathbf{\bar{v}})$ and 
${\cal M}_z^+(\mathbf{\bar{v}})$, to be updated in every 
iteration, are constructed to include the indices of the 
largest $p \times 100 \% $ members of 
$\{\, \bar{y}_i \, \mid \bar{y}_i > 0 \,\}$ and the indices 
of the largest $p \times 100 \% $ members of 
$\{\, \bar{z}_i \, \mid \bar{z}_i > 0 \,\}$ correspondingly, 
determining in particular the integer variables to be fixed 
in the reduced-MILP($\mathbf{\bar{v}}$) for the
LPCC and similarly in the restricted-KKT systems
(\ref{eq:MILP restricted KKT}). 

With a new solution obtained from solving the current partial 
MILP, we compare the objective value of the new solution 
against the previous one.  We maintain the current value of $p$ 
(the proportion of integer variables to fix) if a decrease
in objective value is observed and $p$ has not been kept 
unchanged for $r_{\max}$ iterations. Otherwise, we 
decrease $p$ by the specified amount $\alpha = 0.1$. Note that, 
similar to the time limit for solving each sub-problem
(see below), the maximum allowed repeats $r_{\max}$ and the 
expanding step size $\alpha$ are hyperparameters that can be 
tuned as well. A higher $r_{\max}$ value encourages more 
exploitation but less exploration, while a larger $\alpha$ leads 
to the opposite effect. Hence, if each 
sub-problem can be solved relatively fast, PIP computation can 
benefit from a higher $r_{\max}$ and smaller $\alpha$ to achieve 
faster improvement of solutions through strengthened 
exploitation. With the proportion $p$ updated and a new LPCC 
feasible solution obtained, we reconstruct the fixing index 
sets and solve the new partial MILP (with warm start) from 
the updated solution. This process is repeated until the maximum 
proportion of integer variables to solve is reached, 
i.e., when $p < p_{\min}$; when the latter occurs, we
terminate return the final solution.

\gap

\noindent {\bf Time limits, early termination and inexact solutions:}
When the computational budget is limited, it is important to allocate resources effectively across all subproblems, which can be done by carefully adjusting GUROBI's time-limit parameter. The PIP method benefits from balancing exploration--solving more partial MILPs with varying sizes and different fixed variables-- and exploitation, which involves dedicating more time to seek a favorable solution to a single subproblem. Increasing the computation time for each subproblem promotes deeper exploitation, while reducing it allows for more subproblems to be solved, enhancing exploration. 
In our experiments, based on 
observations from some preliminary
runs, we set the time limit for solving 
each reduced-MILP 
(\ref{eq:partial MILP}) in the PIP 
procedure to be 10 minutes, except for
those (\ref{eq:MILP restricted KKT}) in the
QAPs with $n < 30$ for which
we set the time limit to be 1 minute.  Two
consequences are worth mentioning.  One: 
in spite of the early termination of the
PIP subproblems (solved by GUROBI), all
generated iterates are feasible to the 
LPCC (\ref{eq:LPCC}); this is 
because (i) an initial iterate for PIP can be 
generated by the procedure described above, 
and (ii) GUROBI always outputs a
feasible IP solution at (early or natural) 
termination when there is one. Thus, the 
objective value at termination of PIP is 
always an upper bound of the optimal objective 
value of the LPCC being solved.  Two: there is
a possibility that at PIP termination 
Proposition~\ref{pr:termination of PIP}
may not be applicable when the last subproblem
is early terminated without obtaining
an optimal solution.  In these cases,
the quality of the computed solutions
at termination of PIP is assessed by other
means; see the summary (Table~\ref{tab:stqp_summary}
in particular) the end of the next subsection
for the StQP.  Computationally the solutions obtained by the solvers are subject to numerical tolerance. In particular, the computed solutions may be slightly infeasible to the subproblems. However, such inprecision does not seem to affect the success of PIP.  A summary of all the numerical results
is presented in Subsection~\ref{subsec:summary}.

\gap

\noindent {\bf Computing platform:}
The experiments for the StQP, inverse QP, and inverse AVI instances are 
conducted locally on a MacBook Pro 2020 laptop with an 
Apple M1 3.2 GHz processor and 16 GB of memory. For the 
more computationally demanding QAP instances, a Windows desktop 
with a Ryzen 7 5800X 3.8 GHz processor and 16 GB of memory is 
utilized.  All computations are coded in Python 3.12 with Pyomo,
Numpy, and Scipy packages.

\subsection{Problem instances} \label{sec: instances}

In what follows, we detail the four categories of 
problem instances employed in the experiments, along with
the upper bounds for the complementary variables in the 
resulting LPCCs.

\gap

\noindent $\bullet $ {\bf Standard quadratic program.}  
Coined by 
\cite{Bomze98}, the standard quadratic program  (StQP)
\begin{equation} \label{eq:sqp}
\mbox{StQP} \ \left\{ \, \begin{array}{ll}
\displaystyle{
\operatornamewithlimits{\mbox{\bf minimize}}_{x \in \mathbb{R}^n}
} & \frac{1}{2}x^{\top}Qx + c^{\top}x \\ [0.1in]
\mbox{\bf subject to} & \displaystyle{
\sum_{i=1}^n
} \, x_i = 1 \ \mbox{\bf and} \ 
x_i \geq 0, \quad i = 1, \ldots, n
\end{array} \right\},
\end{equation}
is a fundamental problem in nonlinear programming and critical 
in various applications. Following 
Sections~\ref{sec:LPCC} and \ref{sec:MILP formulations}, a 
full MILP reformulation for the StQP \eqref{eq:sqp} is given by 
\begin{equation} \label{eq:lpcc for sqp}
\begin{array}{ll}
\displaystyle{
\operatornamewithlimits{\mbox{\bf minimize}}_{
(x,z,\lambda,\mu) \in \mathbb{R}^{3n+1}}
} & \frac{1}{2}(c^{\top}x + \mu) \\ [0.1in]
\mbox{\bf subject to} & Qx - \lambda - \mu \onebld = 
-c \\ [0.1in]
& \onebld ^\top x = 1 \\ [0.1in]
& 0 \leq \lambda_j \leq Mz_j \hspace{0.55in} 
\quad j = 1, \ldots, n \\ [0.1in]
& 0 \leq x_j \leq (1 - z_j) \hspace{0.35in} \quad 
j = 1, \ldots, n \\ [0.1in]
\mbox{\bf and} & z_j \in \{0,1\}, \hspace{0.72in} 
\quad j = 1, \ldots, n
\end{array} 
\end{equation}
where $\onebld$ is the vector
of all ones and
$M=2n \max\limits_{i,j \in \{1, \ldots, n\}} |Q_{ij}|$ 
in \eqref{eq:lpcc for sqp}.  This bound is based on  
\cite[Theorem~1 and Proposition~1]{XiaVeraZuluaga20}, where
a detailed discussion on its derivation can be found. 
The paper \cite{ChenBurer12} also provides a detailed 
discussion of such bounds for the QP dual variable.  
In the experiments, we generate five random StQP instances for 
each combination of varying parameters 
$n \in \{200,500,1000,2000\}$ using the same method as in 
\cite{Nowak99, ScozzariTardella08,XiaVeraZuluaga20}; 
following \cite{Nowak99}, we employ a ``density'' parameter 
$\rho \in \{ 0.5, 0.75 \}$ in 
the construction of the matrix $Q$.  
This density controls the portion of the edges of the simplex along which the quadratic objective is convex.
Indirectly, $\rho$ also controls the difficulty of the problem. The larger $\rho$ is, the more difficult the problem is.
More formal and detailed discussion can be found in \cite{Nowak99}.  Also, the vector $c$ is set to zero in all 
instances as in the literature. 

\gap

\noindent $\bullet $ {\bf Quadratic assignment problem.}  
In a quadratic assignment problem (QAP), one aims to find the optimal assignment of a set of facilities to a set of locations. It has a variety of applications including facility layout, scheduling and economics. Since its introduction in 1958 by Koopmans and Beckmann  \cite{KoopmansBeckmann57}, QAP has been known to be one of the most challenging NP-hard problems. To put it formally,  let $n$ denote the number of facilities (and locations), $F_{ij}$  the flow between facilities $i$ and $j$,  and $D_{kl}$ the distance between locations $k$ and $l$. Define the assignment decision variables $\{ x_{ij} \}_{i,j=1}^n$, where $x_{ij}$ equals 1 if facility $i$ is assigned to location $j$, and 0 otherwise. Then  QAP can be formulated as a binary program
\begin{equation}\label{eq:qap_binary}
\underset{x\in \{0,1\}^{n^2}}{\bf minimize} \, \thalf 
\sum_{i=1}^n \sum_{j=1}^n \sum_{k=1}^n \sum_{l=1}^n 
F_{ij} D_{kl} x_{ik} x_{jl},    
\end{equation}
where the objective represents the total cost of transporting the product flows through the distances associated with the assigned locations.

To transform the binary program \eqref{eq:qap_binary} into a continuous one, we adopt the following strategy.  Define   $c_{ijkl} \triangleq F_{ij} \cdot D_{kl}$ and two matrices $S$,  $Q\in\R^{n^2\times n^2}$ by letting $S_{(i-1)n+k, (j-1)n+l} = c_{ijkl}$ and $Q = S - \alpha I$, where $\alpha > \max\limits_{i} \left(\sum\limits_{j} |S_{ij}|\right)$ guarantees the negative definiteness of the matrix $Q$. Introducing the vector 
$x \triangleq ( x_{ij} )_{i,j=1}^n$, the QAP can be formulated as the following quadratic program of the form~\eqref{eq:forward QP}:
\begin{equation}\label{eq:qap_qp}
\begin{array}{lll}
\displaystyle{
\operatornamewithlimits{\mbox{\bf minimize}}_{
\mathbf{x} \in \mathbb{R}^{n^2}}
} & \thalf \, x^{\top} Q x \\ [0.1in]
\mbox{\bf subject to} & \displaystyle{
\sum_{i=1}^n
} \, x_{ij} = 1, & j = 1, \ldots, n \\ [0.2in]
& \displaystyle{
\sum_{j=1}^n
} \, x_{ij} = 1, & i = 1, \ldots, n, \\ [0.2in]
\mbox{\bf and} & x_{ij} \geq 0, & i, j = 1, \ldots, n,
\end{array}
\end{equation}
where the matrix $Q$ is not positive semidefinite. 
The equivalence between \eqref{eq:qap_binary} and 
\eqref{eq:qap_qp} is shown in \cite{BazaraaSherali82}.
The derivation of the upper bounds for the dual variables
complementary to the primal $x$-variables follows that in
\cite{XiaVeraZuluaga20}, resulting in the bound of
$2n^2$ times the largest absolute value of the elements of $Q$.

\gap

\noindent   We evaluate the performance of PIP and FMIP methods on 20 QAP instances with various sizes from the QAPLIB library \cite{BurkardKarischRendl97}. These instances include \textit{esc16b},  \textit{nug18}, \textit{nug20}, \textit{tai20a}, \textit{nug21}, \textit{nug22}, \textit{nug24}, \textit{nug25}, \textit{tai25a}, \textit{bur26a}, \textit{bur26b}, \textit{bur26c}, \textit{bur26d}, \textit{nug27}, \textit{nug28}, \textit{tho30}, \textit{tai30a}, \textit{tai35a}, \textit{tho40}, \textit{tai40a}, where the numbers indicate $n$. The \textit{nug} series of instances, named after the first
author of \cite{Nugent68}, are particularly noted for their use in evaluating facility location assignments. The \textit{bur} series, named after Burkard, the first author of the QAPLIB publication \cite{BurkardKarischRendl97}, models the problem of designing efficient typing keyboards \cite{Burkard1977}. The \textit{tho} instances are provided by Thonemann and B\"olte in \cite{ThonemannBolte1994}. The \textit{esc} instance is derived from practical problems such as the testing of sequential circuits in computer science and backboard wiring problems \cite{EsWu90}. And the \textit{tai} series are uniformly generated and were proposed in \cite{Taillard91}. These instances have been widely used for assessing the efficiency of QAP solving algorithms in literature. 

\gap

\noindent $\bullet $ {\bf Inverse quadratic program.} 
As described in Section~\ref{sec:LPCC},
inverse (convex) quadratic programs (InvQP) can be formulated as 
LPCC (\ref{eq:inverse QP}) naturally. In our experiments, 10 
inverse QP instances with 2 different sizes are randomly 
generated following the scheme proposed in 
\cite[Section~6.3]{JMoroniPangWaechter18}, with a 
modification on the fixed sparsity to a higher 50\% for both $Q$ 
and $D$. Specifically, 5 instances have $m=200$ and $n=150$ 
hence 200 integer variables in total, while the other 5 
instances have $m=1000$ and $n=750$ hence 1000 integer variables 
in the corresponding full MILPs.  The upper bounds of the
complementary variables in the LPCC are provided 
in \cite[Section~6.3]{JMoroniPangWaechter18}.

\gap

\noindent $\bullet $ {\bf Inverse 
affine equilibrium problem.}
The forward problem in this group is an equilibrium problem
modeled by the simple affine
variational inequality (AVI) in the variable $y$ 
parameterized by a bounded first-level design variable 
$x$ in an equality
constraint: find $\bar{y} \in P(x)$ such that
\[
( \, q + Q \bar{y} \, )^{\top}( \bar{y} - y ) \, \geq \, 0,
\epc \forall \, y \, \in \, P(x) \, \triangleq \, 
\left\{ \, y \, \in \, \mathbb{R}^m_+ \, \mid \, 
\onebld^{\top} y \, = \, 1 + a^{\top}x \, \right\}.
\]
Each variable $x_i$ is subject to upper ($u_x$) and lower 
($\ell_x$) bound.  By grouping together $x$ with the 
multiplier $\mu$ of the equality constraint 
$\onebld^{\top} y \, = \, 1 + a^{\top}x$, we can 
easily write the inverse of this AVI in the form of the 
LPCC (\ref{eq:LPCC with x constraint}) with additional
constraints on $x$:
\begin{equation} \label{eq:LPCC in InvAVI}
\begin{array}{ll}
\displaystyle{
\operatornamewithlimits{\mbox{\bf minimize}}_{(x,y,z) 
\in \mathbb{R}^{n + 2m}}
} & \theta \, \triangleq \, 
c^{\top}x + e^{\top}y \\ [0.1in]
\mbox{\bf subject to} & \left[ \begin{array}{cc}
0 & 1 \\ [5pt]
-a^{\top} & 0
\end{array} \right] \left( \begin{array}{c}
x \\ [5pt]
\mu
\end{array} \right) + \left[ \begin{array}{l}
Q \\ [5pt]
\onebld^{\top}
\end{array} \right] y - \left[ \begin{array}{l}
\mathbb{I} \\ [5pt]
0
\end{array} \right] z \, = \, \left( 
\begin{array}{c}
-q \\ [5pt]
1
\end{array} \right); \
\ell_x \, \onebld \, \leq \, x \, \leq \,  u_x \, \onebld
\\ [0.3in]
\mbox{\bf and} & 0 \, \leq \, y \, \perp \, z \, \geq \, 0.
\end{array} \end{equation}
In the experiments, we randomly generate an \emph{asymmetric} matrix
$Q \in \mathbb{R}^{m \times m}$ and vectors $c \in \mathbb{R}^n$
and $e \in \mathbb{R}^m$.  Next we generate a vector
$a \in \mathbb{R}^n_+$, and an initial $x^0$ satisfying
the bounds which we set to be $\ell_x = 0.1$ and $u_x = 1$.
We generate an initial $z^0$ with half of its components
equal to zero.  For the zero components of $z^0$, 
we set the corresponding (nonzero) components of $y^0$
all equal to $\displaystyle{
\frac{m}{2}
} \, ( 1 + a^{\top}x^0 )$; we then let
$q = -Qy^0 +z^0$.  Thus the triplet $(x^0,y^0,z^0)$
is feasible to (\ref{eq:LPCC in InvAVI}).  The upper bounds
for $y$ and $z$ are set to be $u_y = 1 + n \max( a_i )u_x$ and
$u_z = m \, u_y \, \displaystyle{
\max_{ij}
} \, (Q_{ij}) + \displaystyle{
\max_i
} \, q_i$, respectively.

 Similar to the previous problem sets, to test the performance 
and scalability of the PIP method, we include 5 groups of 
5 InvAVI instances with different sizes: the number of 
complementarity pairs 
$m \in \{ 100, 150, 250, 300, 350\}$, and the number 
of the $x$-variable 
is always $\frac{m}{4}$ rounded to the nearest integer.

\subsection{Numerical results}

Throughout the analysis of the results, 
``average'' refers to ``arithmetic average''.  We include 
results from the full MILP solution approach (FMIP) and the warm started full MILP approach (FMIP-W) on the same instances for comparison with the PIP method. In the FMIP approach, problems are reformulated into \eqref{eq:FIP of LPCC} and solved by the MILP solver directly, under a time limit comparable to the average longest computation time of the PIP methods on instances with similar size and complexity. The FMIP approach shares the same essential procedure with the solution method described in \cite{XiaVeraZuluaga20}, especially on QP, where it showcases impressive performance advantage over other popular approaches as tested in the paper. And in the FMIP-W approach, there is an extra step included at the beginning of the FMIP procedure to produce the same feasible solution as the PIP methods for the MIP solver to warm start with, which in comparison can help us further understand the capability of solution improvement of the PIP method. 

\gap 

To evaluate the performance of the proposed PIP method, StQP, QAP, InvQP, and InvAVI instances described in Section \ref{sec: instances} are solved using the PIP, FMIP, and FMIP-W methods implemented as described in Section \ref{sec: implementation}. In the numerical results, we record the best objective value of the solutions produced by all methods on the same instance, together with the corresponding total running time spent in seconds to evaluate the quality of solutions produced by each method in comparable time. Note that for a fair comparison and a better understanding of what to expect in practical usage, such total time is the entire computational time from initialization to (potentially early) termination, including the time spent on any problem formulation and solving for initial solutions. In 
particular, the reported IPOPT time is the time this solver takes to obtain the initial solution used to start the PIP method; the same applies to other initialization methods. The information in the parenthesis following ``FMIP'' and ``FMIP-W'' in the tables indicates the time limit we provided to the MILP solver, and in the case where the FMIP or FMIP-W method terminates with global optimality, we emit the time limit and denote with ``(Opt)'' instead.  Whereas for the PIP methods, the $p_{\max}$ values in the parenthesis indicate the maximum proportion of unfixed integer variables as described in Section~\ref{sec:PIP}. 

\gap

For the StQP instances, we also use GUROBI (version 11.0.1) and BARON (version 24.12.21) as the QP global solvers for directly solving the continuous quadratic programs in its original form. 
 We believe that these results provide a good addition to \cite{XiaVeraZuluaga20}, where it is demonstrated that solving QPs as a MILPs to \emph{certified global} optimality is significantly more efficient than solving them in its original continuous form. In addition, because many heuristic techniques, such as multi-start and local search, are integrally embedded in {\sc Baron} for solving the problem, our results also indirectly illustrate how PIP improves solutions compared with heuristic methods. Both global QP solvers are in default settings and given one hour of solver time limit regardless of problem size.

\gap

\noindent $\bullet $ {\bf StQP results.}
Starting with StQP results presented in Table \ref{tab:sqp_200_05} to \ref{tab:sqp_2000_075}, we can see that, for StQP instances, PIP consistently produces solutions with the high quality and meanwhile demand short or at least comparable solution time in contrast to the other methods. Moreover, the superiority of PIP on improving solution quality becomes even more prominent when the number of integer variables and density increase. To be more precise, we discuss the results under each specific setting to showcase how well PIP adapts to problem size and complexity. 

\gap 

Table \ref{tab:sqp_200_05}, \ref{tab:sqp_200_075}, and \ref{tab:sqp_500_05} present the objective values and running times for each method under easy StQP settings. In these scenarios, 
FMIP successfully solves all instances to provable global optimality, allowing us to measure the improvement of our new method over IPOPT. For example, with $n=200$ and $\rho=0.75$, PIP (0.8) reduces the optimality gap by at least 50\% on average in 15 seconds. Additionally, it is worth noting that PIP (0.9) consistently produces globally optimal solutions for all instances. These results align with Proposition~\ref{pr:termination of PIP} and the remarks that follow.  PIP also outperforms the global QP solvers significantly in these instances.
\begin{table}[!h]
\centering
\caption{StQP, $n=200$, $\rho=0.5$}
\label{tab:sqp_200_05}
\small
\begin{tabular}{
  @{}l 
  S[table-format=-1.2] 
  S[table-format=4.0]  
  S[table-format=-1.2] 
  S[table-format=4.0]  
  S[table-format=-1.2] 
  S[table-format=4.0]  
  S[table-format=-1.2] 
  S[table-format=4.0]  
  S[table-format=-1.2] 
  S[table-format=4.0]  
  @{} 
}
\toprule
 & \multicolumn{2}{c}{1} & \multicolumn{2}{c}{2} & \multicolumn{2}{c}{3} & \multicolumn{2}{c}{4} & \multicolumn{2}{c}{5} \\
\cmidrule(lr){2-3} \cmidrule(lr){4-5} \cmidrule(lr){6-7} \cmidrule(lr){8-9} \cmidrule(l){10-11}
{Method} & {Obj} & {Time} & {Obj} & {Time} & {Obj} & {Time} & {Obj} & {Time} & {Obj} & {Time} \\
\midrule
IPOPT         & -3.07 & 4    & -3.17 & 4    & -3.19 & 4    & -3.16 & 4    & -2.89 & 4    \\
BARON (1hr)   & -3.09 & 3606 & -3.42 & 3607 & -3.26 & 3607 & -3.36 & 3606 & -3.36 & 3608 \\
GUROBI (1hr)    & -3.07 & 3633 & -3.17 & 3633 & -3.19 & 3633 & -3.16 & 3633 & -2.91 & 3633 \\
FMIP (Opt)    & -3.22 & 41   & -3.42 & 14   & -3.31 & 14   & -3.36 & 13   & -3.36 & 13   \\
FMIP-W (1min) & -3.22 & 19   & -3.42 & 17   & -3.31 & 18   & -3.36 & 17   & -3.36 & 17   \\
PIP (0.4)     & -3.22 & 10   & -3.17 & 9    & -3.19 & 9    & -3.16 & 10   & -2.89 & 10   \\
PIP (0.6)     & -3.22 & 10   & -3.17 & 10   & -3.19 & 10   & -3.36 & 11   & -2.89 & 10   \\
PIP (0.8)     & -3.22 & 12   & -3.18 & 12   & -3.19 & 11   & -3.36 & 12   & -3.14 & 14   \\
PIP (0.9)     & -3.22 & 18   & -3.42 & 24   & -3.31 & 24   & -3.36 & 18   & -3.36 & 25   \\
\bottomrule
\end{tabular}
\end{table}

\begin{table}[!h]
\centering
\caption{StQP, $n=200$, $\rho=0.75$}
\label{tab:sqp_200_075}
\small
\begin{tabular}{
  @{}l 
  S[table-format=-1.2] 
  S[table-format=4.0]  
  S[table-format=-1.2] 
  S[table-format=4.0]  
  S[table-format=-1.2] 
  S[table-format=4.0]  
  S[table-format=-1.2] 
  S[table-format=4.0]  
  S[table-format=-1.2] 
  S[table-format=4.0]  
  @{} 
}
\toprule
 & \multicolumn{2}{c}{1} & \multicolumn{2}{c}{2} & \multicolumn{2}{c}{3} & \multicolumn{2}{c}{4} & \multicolumn{2}{c}{5} \\
\cmidrule(lr){2-3} \cmidrule(lr){4-5} \cmidrule(lr){6-7} \cmidrule(lr){8-9} \cmidrule(l){10-11}
{Method} & {Obj} & {Time} & {Obj} & {Time} & {Obj} & {Time} & {Obj} & {Time} & {Obj} & {Time} \\
\midrule
IPOPT          & -3.27 & 4    & -3.34 & 4    & -3.09 & 5    & -3.25 & 4    & -3.22 & 4    \\
BARON (1hr)    & -3.29 & 3606 & -3.45 & 3732 & -3.43 & 3603 & -3.38 & 3607 & -3.33 & 3608 \\
GUROBI (1hr)     & -3.27 & 3633 & -3.34 & 3633 & -3.30 & 3633 & -3.30 & 3633 & -3.36 & 3633 \\
FMIP (Opt)     & -3.35 & 107  & -3.50 & 111  & -3.43 & 74   & -3.41 & 126  & -3.37 & 371  \\
FMIP-W (1min)  & -3.31 & 70   & -3.50 & 70   & -3.43 & 37   & -3.41 & 35   & -3.37 & 51   \\
PIP (0.4)      & -3.27 & 10   & -3.34 & 10   & -3.10 & 11   & -3.25 & 10   & -3.22 & 10   \\
PIP (0.6)      & -3.27 & 10   & -3.34 & 11   & -3.25 & 12   & -3.30 & 11   & -3.35 & 11   \\
PIP (0.8)      & -3.27 & 13   & -3.42 & 15   & -3.30 & 15   & -3.35 & 13   & -3.35 & 13   \\
PIP (0.9)      & -3.35 & 74   & -3.50 & 53   & -3.43 & 65   & -3.41 & 65   & -3.37 & 78   \\
\bottomrule
\end{tabular}
\end{table}
\begin{table}[!h]
\centering
\caption{StQP, $n=500$, $\rho=0.5$}
\label{tab:sqp_500_05}
\small
\begin{tabular}{
  @{}l 
  S[table-format=-1.2] 
  S[table-format=4.0]  
  S[table-format=-1.2] 
  S[table-format=4.0]  
  S[table-format=-1.2] 
  S[table-format=4.0]  
  S[table-format=-1.2] 
  S[table-format=4.0]  
  S[table-format=-1.2] 
  S[table-format=4.0]  
  @{} 
}
\toprule
 & \multicolumn{2}{c}{1} & \multicolumn{2}{c}{2} & \multicolumn{2}{c}{3} & \multicolumn{2}{c}{4} & \multicolumn{2}{c}{5} \\
\cmidrule(lr){2-3} \cmidrule(lr){4-5} \cmidrule(lr){6-7} \cmidrule(lr){8-9} \cmidrule(l){10-11}
{Method} & {Obj} & {Time} & {Obj} & {Time} & {Obj} & {Time} & {Obj} & {Time} & {Obj} & {Time} \\
\midrule
IPOPT         & -3.17 & 30   & -3.32 & 30   & -3.28 & 33   & -3.24 & 32   & -3.25 & 31   \\
BARON (1hr)   & -3.39 & 3620 & -3.37 & 3623 & -3.32 & 3621 & -3.38 & 3621 & -3.29 & 3620 \\
GUROBI (1hr)    & -3.17 & 3649 & -3.32 & 3652 & -3.24 & 3649 & -3.24 & 3649 & -3.26 & 3649 \\
FMIP (Opt)    & -3.48 & 879  & -3.54 & 2409 & -3.43 & 1111 & -3.42 & 883  & -3.48 & 1370 \\
FMIP-W (10min)& -3.48 & 664  & -3.47 & 663  & -3.43 & 666  & -3.42 & 665  & -3.48 & 664 \\
PIP (0.4)     & -3.17 & 64   & -3.38 & 65   & -3.28 & 67   & -3.27 & 67   & -3.25 & 65 \\
PIP (0.6)     & -3.32 & 70   & -3.38 & 68   & -3.28 & 72   & -3.27 & 71   & -3.28 & 70 \\
PIP (0.8)     & -3.32 & 86   & -3.38 & 81   & -3.43 & 504  & -3.33 & 92   & -3.36 & 95 \\
PIP (0.9)     & -3.48 & 687  & -3.54 & 596  & -3.43 & 1105& -3.42 & 693  & -3.48 & 636 \\
\bottomrule
\end{tabular}
\end{table}

\gap

On our machine, we notice that FMIP is no longer able to provide global optimality in one hour for the more complex instances with $n=500$ and $\rho=0.75$. On the contrary, PIP (0.4 / 0.6 / 0.8) start to produce solutions with matching or strictly better quality in a fraction of time compared to FMIP and FMIP-W, which is shown in Table \ref{tab:sqp_500_075}, as the best solution quality maintained by PIP (0.9) in a comparable computation time of 10 minutes. An interesting observation is that the IPOPT (NLP approach) is able to produce better solutions than 
GUROBI (FMIP approach)
on some instances in Table \ref{tab:sqp_500_075}, which does not happen on the smaller or less complex instances. However, such increase in problem size and complexity does not impact PIP solution quality as much, suggesting the advantage in scalability of PIP. Another notable observation is that on instance 2 of Table \ref{tab:sqp_500_075},  IPOPT yields the best solution already, which we suspect is optimal since there is still no better solution found after another exhaustive FMIP-W run for 3 hours, marking the sole instance in our entire experiments where PIP does not provide strict improvement to the solution provided by IPOPT. As for the comparison with the one-hour approaches, PIP solutions still outperform GUROBI, but BARON is catching up in solution quality. We also explore the performance of FMIP given a one-hour computational time to see if global optimality can be obtained, which does not happen in any instance, similar to what we see in the global solver results, although it does lead to FMIP to provide solutions with better solution quality. 

\begin{table}[!h]
\centering
\caption{StQP, $n=500$, $\rho=0.75$}
\label{tab:sqp_500_075}
\small
\begin{tabular}{
  @{}l 
  S[table-format=-1.2] 
  S[table-format=4.0]  
  S[table-format=-1.2] 
  S[table-format=4.0]  
  S[table-format=-1.2] 
  S[table-format=4.0]  
  S[table-format=-1.2] 
  S[table-format=4.0]  
  S[table-format=-1.2] 
  S[table-format=4.0]  
  @{} 
}
\toprule
 & \multicolumn{2}{c}{1} & \multicolumn{2}{c}{2} & \multicolumn{2}{c}{3} & \multicolumn{2}{c}{4} & \multicolumn{2}{c}{5} \\
\cmidrule(lr){2-3} \cmidrule(lr){4-5} \cmidrule(lr){6-7} \cmidrule(lr){8-9} \cmidrule(l){10-11}
{Method} & {Obj} & {Time} & {Obj} & {Time} & {Obj} & {Time} & {Obj} & {Time} & {Obj} & {Time} \\
\midrule
IPOPT         & -3.51 & 30   & -3.64 & 31   & -3.51 & 30   & -3.52 & 30   & -3.45 & 35   \\
BARON (1hr)   & -3.62 & 3626 & -3.50 & 3620 & -3.45 & 3621 & -3.61 & 3621 & -3.58 & 3621 \\
GUROBI (1hr)    & -3.51 & 3649 & -3.64 & 3649 & -3.51 & 3649 & -3.57 & 3649 & -3.45 & 3649 \\
FMIP (10min)  & -3.47 & 632  & -3.43 & 632  & -3.50 & 632  & -3.52 & 632  & -3.58 & 632  \\
FMIP (1hr)    & -3.62 & 3665 & -3.64 & 3666 & -3.56 & 3666 & -3.63 & 3665 & -3.59 & 3670 \\
FMIP-W (10min)& -3.55 & 664  & -3.64 & 664  & -3.51 & 664  & -3.59 & 663  & -3.53 & 668  \\
PIP (0.4)     & -3.52 & 66   & -3.64 & 65   & -3.51 & 65   & -3.52 & 65   & -3.45 & 70   \\
PIP (0.6)     & -3.52 & 70   & -3.64 & 69   & -3.51 & 69   & -3.54 & 71   & -3.53 & 77   \\
PIP (0.8)     & -3.58 & 124  & -3.64 & 86   & -3.63 & 102  & -3.61 & 107  & -3.53 & 101  \\
PIP (0.9)     & -3.58 & 723  & -3.64 & 687  & -3.64 & 702  & -3.61 & 707  & -3.58 & 701  \\
\bottomrule
\end{tabular}
\end{table}
\begin{table}[!h]
\centering
\caption{StQP, $n=1000$, $\rho=0.5$}
\label{tab:sqp_1000_05}
\small
\begin{tabular}{
  @{}l 
  S[table-format=-1.2] 
  S[table-format=4.0]  
  S[table-format=-1.2] 
  S[table-format=4.0]  
  S[table-format=-1.2] 
  S[table-format=4.0]  
  S[table-format=-1.2] 
  S[table-format=4.0]  
  S[table-format=-1.2] 
  S[table-format=4.0]  
  @{} 
}
\toprule
 & \multicolumn{2}{c}{1} & \multicolumn{2}{c}{2} & \multicolumn{2}{c}{3} & \multicolumn{2}{c}{4} & \multicolumn{2}{c}{5} \\
\cmidrule(lr){2-3} \cmidrule(lr){4-5} \cmidrule(lr){6-7} \cmidrule(lr){8-9} \cmidrule(l){10-11}
{Method} & {Obj} & {Time} & {Obj} & {Time} & {Obj} & {Time} & {Obj} & {Time} & {Obj} & {Time} \\
\midrule
IPOPT         & -3.36 & 166  & -3.42 & 160  & -3.50 & 156  & -3.35 & 164  & -3.16 & 170  \\
BARON (1hr)   & -3.51 & 3670 & -3.41 & 3671 & -3.51 & 3671 & -3.53 & 3674 & -3.47 & 3675 \\
GUROBI (1hr)    & -3.39 & 3702 & -3.40 & 3702 & -3.38 & 3703 & -3.42 & 3702 & -3.49 & 3703 \\
FMIP (10min)  & -3.51 & 725  & -3.41 & 728  & NA    & 730  & -3.37 & 727  & -3.46 & 729  \\
FMIP (1hr)    & -3.60 & 3898 & -3.57 & 3901 & -3.47 & 3897 & -3.52 & 3904 & -3.49 & 3909 \\
FMIP-W (10min)& -3.50 & 893  & -3.42 & 891  & -3.50 & 887  & -3.52 & 894  & -3.40 & 901  \\
PIP (0.4)     & -3.36 & 300  & -3.42 & 296  & -3.50 & 292  & -3.46 & 304  & -3.18 & 310  \\
PIP (0.6)     & -3.44 & 325  & -3.42 & 313  & -3.50 & 308  & -3.46 & 318  & -3.51 & 343  \\
PIP (0.8)     & -3.55 & 502  & -3.48 & 532  & -3.55 & 469  & -3.46 & 414  & -3.51 & 469  \\
\bottomrule
\end{tabular}
\end{table}
\begin{table}[!h]
\centering
\caption{StQP, $n=1000$, $\rho=0.75$}
\label{tab:sqp_1000_075}
\small
\begin{tabular}{
  @{}l 
  S[table-format=-1.2] 
  S[table-format=4.0]  
  S[table-format=-1.2] 
  S[table-format=4.0]  
  S[table-format=-1.2] 
  S[table-format=4.0]  
  S[table-format=-1.2] 
  S[table-format=4.0]  
  S[table-format=-1.2] 
  S[table-format=4.0]  
  @{} 
}
\toprule
 & \multicolumn{2}{c}{1} & \multicolumn{2}{c}{2} & \multicolumn{2}{c}{3} & \multicolumn{2}{c}{4} & \multicolumn{2}{c}{5} \\
\cmidrule(lr){2-3} \cmidrule(lr){4-5} \cmidrule(lr){6-7} \cmidrule(lr){8-9} \cmidrule(l){10-11}
{Method} & {Obj} & {Time} & {Obj} & {Time} & {Obj} & {Time} & {Obj} & {Time} & {Obj} & {Time} \\
\midrule
IPOPT         & -3.55 & 178  & -3.49 & 180  & -3.68 & 176  & -3.50 & 169  & -3.41 & 179 \\
BARON (1hr)   & -3.70 & 3674 & -3.59 & 3687 & -3.60 & 3670 & -3.61 & 3672 & -3.60 & 3682 \\
GUROBI (1hr)    & -3.49 & 3703 & -3.47 & 3703 & -3.53 & 3703 & -3.56 & 3703 & -3.61 & 3703 \\
FMIP (10min)  & -3.46 & 727  & -3.53 & 728  & -3.53 & 727  & NA    & 727  & NA    & 728 \\
FMIP (1hr)    & -3.52 & 3912 & -3.61 & 3917 & -3.61 & 3914 & NA    & 3907 & -3.14 & 3919 \\
FMIP-W (10min)& -3.55 & 910  & -3.55 & 910  & -3.68 & 908  & -3.50 & 900  & -3.41 & 910 \\
PIP (0.4)     & -3.55 & 312  & -3.53 & 322  & -3.68 & 310  & -3.50 & 302  & -3.41 & 312 \\
PIP (0.6)     & -3.60 & 347  & -3.60 & 384  & -3.68 & 332  & -3.62 & 353  & -3.54 & 379 \\
PIP (0.8)     & -3.63 & 835  & -3.60 & 645  & -3.73 & 711  & -3.62 & 696  & -3.66 & 1154 \\
\bottomrule
\end{tabular}
\end{table}

\gap

When the size of the StQP is increased further to $n=1000$, the advantage of PIP in solution quality and efficiency extends further. As shown in Table \ref{tab:sqp_1000_05} and \ref{tab:sqp_1000_075}, PIP provides solutions with strictly better quality in shorter time on all instances compared to FMIP, and in 9 out of 10 instances compared to FMIP-W. Especially in the harder instances with $\rho = 0.75$ in Table \ref{tab:sqp_1000_075}, PIP ($0.6$) outperforms FMIP-W with better solution quality consistently while requiring less than half computation time. On the contrary, with 1000 integer variables to solve, FMIP starts to struggle to provide any feasible solutions on some instances in a comparable time. Given one hour, FMIP does eventually improve and provide better quality in 4 out of 10 instances than PIP solutions obtained in 7 to 8 minutes. And compared with solutions obtained from one-hour methods, PIP outperforms GUROBI one-hour solution quality in all instances and BARON one-hour solution quality in 8 out of 10 instances.
\gap

On the largest and most challenging StQP instances with $n=2000$ in Table \ref{tab:sqp_2000_05} and \ref{tab:sqp_2000_075}, PIP maintains the leading quality of solutions in all instances with a shorter computation time compared to FMIP and FMIP-W. Notice that with 2000 integer variables, FMIP struggles further and only finds feasible solutions with subpar quality in 3 out of 10 instances, thus exposing the drawback in scalability of the full MILP approach. On the contrary, PIP continues to excel. On instance 1, 2, 3, and 5 of Table \ref{tab:sqp_2000_075}, where FMIP-W fails to provide any improvement from IPOPT solutions in more than an hour, PIP ($0.8$) is still able to provide considerable improvement in a much shorter time consistently. Compared with global QP solvers' one-hour results, PIP solution quality is similar or better in most instances with shorter computational time, except for the fifth instance in Table \ref{tab:sqp_2000_075}, where BARON one-hour solution quality excels by a large margin. 
\begin{table}[!h]
\centering
\caption{StQP, $n=2000$, $\rho=0.5$}
\label{tab:sqp_2000_05}\small
\begin{tabular}{
  @{}l
  S[table-format=-1.2] 
  S[table-format=4.0]  
  S[table-format=-1.2] 
  S[table-format=4.0]  
  S[table-format=-1.2] 
  S[table-format=4.0]  
  S[table-format=-1.2] 
  S[table-format=4.0]  
  S[table-format=-1.2] 
  S[table-format=4.0]  
  @{}  
}
\toprule
 & \multicolumn{2}{c}{1} & \multicolumn{2}{c}{2} & \multicolumn{2}{c}{3} & \multicolumn{2}{c}{4} & \multicolumn{2}{c}{5} \\
\cmidrule(lr){2-3} \cmidrule(lr){4-5} \cmidrule(lr){6-7} \cmidrule(lr){8-9} \cmidrule(l){10-11}
{Method} & {Obj} & {Time} & {Obj} & {Time} & {Obj} & {Time} & {Obj} & {Time} & {Obj} & {Time} \\
\midrule
IPOPT         & -3.54 & 871  & -3.35 & 894  & -3.43 & 922  & -3.22 & 990  & -3.47 & 969  \\
BARON (1hr)   & -3.63 & 3882 & -3.58 & 3878 & -3.49 & 3875 & -3.56 & 3882 & -3.47 & 3877 \\
GUROBI (1hr)    & -3.44 & 3920 & -3.46 & 3919 & -3.36 & 3922 & -3.36 & 3919 & -3.36 & 3919 \\
FMIP (10min)  & {NA}  & 1118 & {NA}  & 1119 & {NA}  & 1120 & {NA}  & 1119 & {NA}  & 1118 \\
FMIP (1hr)    & {NA}  & 4119 & -3.45 & 4120 & -3.35 & 4121 & -3.22 & 4119 & {NA}  & 4120 \\
FMIP-W (1hr)  & -3.57 & 4991 & -3.56 & 5014 & -3.50 & 5043 & -3.56 & 5116 & -3.47 & 5099 \\
PIP (0.4)     & -3.54 & 1418 & -3.35 & 1433 & -3.46 & 1482 & -3.51 & 1572 & -3.47 & 1518 \\
PIP (0.6)     & -3.54 & 1511 & -3.53 & 1608 & -3.63 & 1700 & -3.51 & 1636 & -3.52 & 1677 \\
PIP (0.8)     & -3.60 & 3398 & -3.57 & 4041 & -3.63 & 2917 & -3.63 & 4069 & -3.63 & 3564 \\
\bottomrule
\end{tabular}
\end{table}

\begin{table}[!h]
\centering
\caption{StQP, $n=2000$, $\rho=0.75$}
\label{tab:sqp_2000_075}
\small
\begin{tabular}{
  @{}l
  S[table-format=-1.2] 
  S[table-format=4.0]  
  S[table-format=-1.2] 
  S[table-format=4.0]  
  S[table-format=-1.2] 
  S[table-format=4.0]  
  S[table-format=-1.2] 
  S[table-format=4.0]  
  S[table-format=-1.2] 
  S[table-format=4.0]  
  @{} 
}
\toprule
 & \multicolumn{2}{c}{1} & \multicolumn{2}{c}{2} & \multicolumn{2}{c}{3} & \multicolumn{2}{c}{4} & \multicolumn{2}{c}{5} \\
\cmidrule(lr){2-3} \cmidrule(lr){4-5} \cmidrule(lr){6-7} \cmidrule(lr){8-9} \cmidrule(l){10-11}
{Method} & {Obj} & {Time} & {Obj} & {Time} & {Obj} & {Time} & {Obj} & {Time} & {Obj} & {Time} \\
\midrule
IPOPT         & -3.59 & 937  & -3.67 & 940  & -3.64 & 1000 & -3.65 & 1362 & -3.59 & 1030 \\
BARON (1hr)   & -3.73 & 4040 & -3.74 & 4042 & -3.64 & 4051 & -3.70 & 4037 & -3.85 & 4033 \\
GUROBI (1hr)    & -3.70 & 3921 & -3.72 & 3920 & -3.54 & 3920 & -3.61 & 3920 & -3.55 & 3921 \\
FMIP (10min)  & {NA}  & 1118 & {NA}  & 1119 & {NA}  & 1119 & {NA}  & 1119 & {NA}  & 1118 \\
FMIP (1hr)    & {NA}  & 4119 & {NA}  & 4119 & {NA}  & 4119 & {NA}  & 4119 & {NA}  & 4118 \\
FMIP-W (1hr)  & -3.59 & 5054 & -3.67 & 5076 & -3.64 & 5111 & -3.68 & 5499 & -3.59 & 5157 \\
PIP (0.4)     & -3.59 & 1489 & -3.67 & 1498 & -3.64 & 1548 & -3.65 & 1912 & -3.59 & 1579 \\
PIP (0.6)     & -3.76 & 2227 & -3.72 & 1927 & -3.64 & 1614 & -3.65 & 1975 & -3.65 & 1843 \\
PIP (0.8)     & -3.76 & 3039 & -3.72 & 2731 & -3.69 & 3130 & -3.69 & 3639 & -3.65 & 2720 \\
\bottomrule
\end{tabular}
\end{table}

\gap

 Comparing the capability to improve solution quality of PIP with different $p_{\max}$ values, we observe that across all StQP instances, except in one instance (with $n=500$ and $\rho=0.75$ when the IPOPT solution is already the best), PIP (0.8 or 0.9) always provides improvements from the initial solutions. We also observe that PIP with a low $p_{\max}$ value (0.4) is the fastest but provides improvement in fewer instances, which illustrates the natural trade-off between the solution quality and computational time for PIP. While with a medium $p_{\max}$ value, PIP ($0.6$) always provides decent improvements over the initial solutions without sacrificing much computational speed. Thus, along with the other adjustable factors discussed in Section \ref{sec: implementation}, $p_{\max}$ is one of the most important hyperparameters in PIP to be fine-tuned in order to achieve the users' preference (or a balance) between solution quality and computational speed. Thus, when prioritizing computational speed, lower values of $p_{\max}$ are recommended. Conversely, higher $p_{\max}$ values should be selected when the focus is on solution quality. To strike a balance between speed and quality, medium values of $p_{\max}$ are advisable. Conducting trial runs with sample instances can provide further insights into the optimal tuning of $p_{\max}$.

\begin{table}[!h]
    \centering
    \caption{StQP: Summary of arithmetic
    average performances (each over 5 instances)}
    \label{tab:stqp_summary}
    \setlength{\tabcolsep}{1.7pt}
    \small
    \begin{tabular}{c|
    S@{\hspace{-0.4em}}r|
    S@{\hspace{-0.4em}}r|
    S@{\hspace{-0.1em}}r|
    S@{\hspace{-0.5em}}r|@{\hspace{-0.5em}}
    S@{\hspace{-0.3em}}r|@{\hspace{-0.5em}}
    S@{\hspace{-0.3em}}r|
    S@{\hspace{-0.1em}}r|
    S@{\hspace{-0.1em}}r|
    c}
        \toprule
        $(n, \rho)$ & \multicolumn{2}{c|}{(200, 0.5)} & \multicolumn{2}{c|}{(200, 0.75)} & \multicolumn{2}{c|}{(500, 0.5)} & \multicolumn{2}{c|@{\hspace{-0.5em}}}{(500, 0.75)} & \multicolumn{2}{c|@{\hspace{-0.5em}}}{\,\,(1000, 0.5)} & \multicolumn{2}{c|}{\,\,(1000, 0.75)} & \multicolumn{2}{c|}{(2000, 0.5)} & \multicolumn{2}{c|}{(2000, 0.75)} & { }\\
        { } & {Stnd} & { } & {Stnd} & { } & {Stnd} & { } & {Stnd} & { } & {Stnd} & { } & {Stnd} & { } & {Stnd} & { } & {Stnd} & { } & \textbf{OBJ}  \\
        {Method} & {Obj} & {Time} & {Obj} & {Time} & {Obj} & {Time} & {Obj} & {Time} & {Obj} & {Time} & {Obj} & {Time} & {Obj} & {Time} & {Obj} & {Time} & \textbf{IMP} \\
        \midrule
IPOPT & 100.00 & 4 & 100.00 & 4 & 95.79 & 31 & 68.51 & 31 & 83.59 & 163 & 64.18 & 176 & 75.61 & 929 & 75.11 & 1054 & -- \\
BARON (1hr) & 25.67 & 3607 & 30.33 & 3631 & 53.81 & 3621 & 38.49 & 3622 & 33.39 & 3672 & 19.64 & 3677 & 25.27 & 3879 & 10.20 & 4041 & -- \\
GUROBI (1hr) & 99.15 & 3633 & 62.73 & 3633 & 99.13 & 3650 & 59.03 & 3649 & 70.87 & 3702 & 69.41 & 3703 & 82.89 & 3920 & 72.77 & 3920 & -- \\
FMIP (Opt) & 0.00 & 19 & 0.00 & 158 & 0.00 & 1330 & {--}  & {--\,\,} & {--} & {--\,\,\,} & {--} & {--\,\,\,} & {--} & {--\,\,\,\,} & {--} & {--\,\,\,\,}   & -- \\
FMIP (10min) & {--} & {--\,\,} & {--} & {--\,\,}  & {--} & {--\,\,}  & 76.37 & 632 & 58.79\tiny{[4]} & 728 & 85.57\tiny{[3]} & 727 & NA & 1119 & NA & 1119 & -- \\
FMIP (1hr) & {--} & {--\,\,} & {--} & {--\,\,}  & {--} & {--\,\,}  & 8.85 & 3667 & 12.63 & 3902 & 58.60\tiny{[4]} & 3914 & 85.51 & 4120 & NA & 4119 & -- \\
FMIP-W & 0.00 & 18 & 10.00 & 53 & 6.36 & 664 & 38.76 & 665 & 39.36 & 893 & 55.52 & 908 & 32.61 & 5053 & 68.44 & 5179 & 57.37\% \\
PIP (0.4) & 80.00 & 10 & 99.41 & 10 & 87.00 & 66 & 67.21 & 66 & 70.23 & 300 & 58.41 & 312 & 59.32 & 1485 & 75.11 & 1605 & 9.77\% \\
PIP (0.6) & 60.00 & 10 & 67.00 & 11 & 74.72 & 70 & 52.12 & 71 & 44.73 & 321 & 19.16 & 359 & 27.82 & 1626 & 36.83 & 1917 & 43.72\% \\
PIP (0.8) & 48.56 & 12 & 47.81 & 14 & 45.30 & 172 & 18.45 & 104 & 22.84 & 477 & 7.07 & 808 & 4.03 & 3598 & 21.27 & 3052 & 69.68\% \\
PIP (0.9) & 0.00 & 22 & 0.00 & 67 & 0.00 & 743 & 10.34 & 704 & {--} & {--\,\,\,} & {--} & {--\,\,\,} & {--} & {--\,\,\,\,} & {--} & {--\,\,\,\,} & 96.23\% \\

        \bottomrule
    \end{tabular}
\end{table}

\gap

\noindent {\bf A summary.}
To provide a summary of the average performance of each 
method on StQP instances and also to gauge 
the quality of the solutions of the methods due to early 
termination that may
invalidate Proposition~\ref{pr:termination of PIP}, 
we standardize the computed objective values to a
range between 0 and 100 with 
\[
\textrm{Standardized Obj} \, \triangleq \, \displaystyle{
\frac{\textrm{Original Obj} - \textrm{min Obj}}{\textrm{max Obj} - \textrm{min Obj}}
} \, \times \, 100, 
\]
where \textrm{max} and \textrm{min Obj} are the maximal 
and minimal objective values achieved by all methods on 
the same instance.  The averaged standardized objective 
values over 5 instances per setting for each method are 
presented in Table~\ref{tab:stqp_summary}, together with 
the corresponding average total computation time.   
Since we are solving minimization problems,
a smaller objective value is desired; hence, after 
standardization, 0 indicates the best while 100 is the 
worst.  This standardized objective is in contrast to 
the quantity
\[
\textrm{OBJ IMP} \, \triangleq \, \displaystyle{
\frac{\textrm{IPOPT Obj} - \textrm{Obj}}{\textrm{IPOPT Obj}-\textrm{min Obj}}
} \, \times \, 100\%,
\]
which is a measure
to quantify the improvement of the objective achieved 
by the PIP and FMIP-W method over the initial ones produced
by traditional 
NLP methods.  The mean value of OBJ IMP  averaged over all 
instances per method is presented in the last column.  
In
terms of this mean OBJ IMP, the improvement quality of the 
method over IPOPT ranges from 0 (worst) to  100 (best). 
If a method fails to solve all 5 instances of the setting, 
we denote the number of instances solved with a 
{\sl feasible} solution in the square brackets, 
and ``NA'' means that the 
method fails on all instances of the setting.  

\gap

From Table~\ref{tab:stqp_summary}, we can see that, for most StQP instances, PIP can produce solutions with the best quality in comparable or shorter time compared to the other methods. Note that in the columns of $(n=500, \rho =0.75)$ and $(n=1000, \rho =0.5)$,  FMIP (1hr) produces better solution quality given five times more computational time on instances. For the last instance setting $(n=2000, \rho=0.75)$, BARON shows exceptional solution quality due to its outstanding performance for the last instance in Table~\ref{tab:sqp_2000_075}. However, aside from this instance, BARON solutions are similar in quality compared to PIP (0.8) given 1000 seconds more computational time on average. In all instances with verifiable global optimality, PIP ($0.9$) always achieves the optimal solutions fast. While in larger instances ($n=1000$ or $2000$), the faster PIP ($0.8$) provides the best or close to the best solution quality with shorter computation time compared to other approaches. Compared with global QP solvers, PIP provides better solution quality in a short time than BARON and GUROBI in most instances. In all instances, PIP effectively improves the solution quality considerably.

\gap

\noindent $\bullet $ {\bf QAP and InvQP results.}
On the QAP instances, in addition to the objective value of best solutions produced by PIP and FMIP methods, we include the best known solutions reported by the QAPLIB database \cite{BurkardKarischRendl97} for a better understanding of the solution quality achieved by PIP under our limited computational budget. Note that the best known solutions are produced by resource-intensive methods designated specifically for QAP, such as robust and parallel adaptive tabu search with supercomputers, see \cite{BruenggerEtAl96, Taillard91, Taillard94, TaHaGe97}. While PIP, applicable to general QP and LPCC, is capable of producing solution with close gap in objective value to the best known solutions on a consumer desktop in a reasonable time.

\gap

Similar to what we have seen in StQP results, the QAP results show that PIP outperforms FMIP in solution quality and computation time on all instances. On the \textit{esc16b} instance, PIP ($0.6$) is able to match the best known solution in about 3 minutes. And as the problem size $n$ increases towards $30$, PIP is still able to provide solutions close to the best known results. In large instances with $n\ge 30$, while FMIP lags behind in solution quality despite being given hours more computation time (4 hours in \textit{tai30a} and \textit{tai35a}; 8 hours in \textit{tai40a}), PIP is able to maintain relative gaps in objective values below $10 \%$ compared to the best known solutions. Note that in QAP instances, PIP ($0.4$) yields slightly lower solution quality compared to PIP ($0.6$) but in significantly less time. This highlights the importance of adjusting the $p_{\max}$ value to align with user preferences.


\begin{table}[!h]
\centering
\caption{QAP Instances Group 1, $n < 30$}
\label{tab:qap_group1}
\small
\begin{tabular}{@{}l
  S[table-format=3.0]@{\hspace{0.9em}} S[table-format=3.0]
  S[table-format=3.0]@{\hspace{0.9em}} S[table-format=3.0]
  S[table-format=3.0]@{\hspace{0.9em}} S[table-format=3.0]
  S[table-format=6.0]@{\hspace{0.9em}} S[table-format=3.0]
  S[table-format=3.0]@{\hspace{0.9em}} S[table-format=3.0]@{}}
\toprule
& \multicolumn{2}{c}{esc16b} 
& \multicolumn{2}{c}{nug18}
& \multicolumn{2}{c}{nug20} 
& \multicolumn{2}{c}{tai20a}
& \multicolumn{2}{c}{nug21} \\
\cmidrule(lr){2-3} \cmidrule(lr){4-5} \cmidrule(lr){6-7} \cmidrule(lr){8-9} \cmidrule(l){10-11}
{Method} & {Obj} & {Time} 
         & {Obj} & {Time} 
         & {Obj} & {Time} 
         & {Obj} & {Time} 
         & {Obj} & {Time} \\
\midrule
Best Known   & 292    &       & 1930   &       & 2570   &       & 703482   &       & 2438   &       \\
FMIP (10min) & 292    & 607   & 2092   & 603   & 2728   & 615   & 744732   & 605   & 2954   & 618   \\
PIP (0.4)    & 310    &  97   & 2134   & 22    & 2758   & 112   & 735928   & 162   & 2568   & 145   \\
PIP (0.6)    & 292    & 196   & 2030   & 265   & 2698   & 355   & 735266   & 500   & 2568   & 328   \\
\bottomrule
\end{tabular}
\end{table}

\begin{table}[!h]
\centering
\caption{QAP Instances Group 2, $n < 30$}
\label{tab:qap_group2}
\small
\begin{tabular}{@{}l
  S[table-format=3.0]@{\hspace{0.9em}} S[table-format=3.0]
  S[table-format=3.0]@{\hspace{0.9em}} S[table-format=3.0]
  S[table-format=3.0]@{\hspace{0.9em}} S[table-format=3.0]
  S[table-format=7.0]@{\hspace{0.9em}} S[table-format=3.0]
  S[table-format=7.0]@{\hspace{0.9em}} S[table-format=3.0]@{}}
\toprule
& \multicolumn{2}{c}{nug22} 
& \multicolumn{2}{c}{nug24} 
& \multicolumn{2}{c}{nug25} 
& \multicolumn{2}{c}{tai25a} 
& \multicolumn{2}{c}{bur26a} \\
\cmidrule(lr){2-3} \cmidrule(lr){4-5} \cmidrule(lr){6-7} \cmidrule(lr){8-9} \cmidrule(l){10-11}
{Method} & {Obj} & {Time} 
         & {Obj} & {Time} 
         & {Obj} & {Time} 
         & {Obj} & {Time} 
         & {Obj} & {Time} \\
\midrule
Best Known   & 3596   &       & 3488   &       & 3744   &       & 1167256  &       & 5426670   &       \\
FMIP (10min / 20min)  & 4160   & 622   & 3874   & 630   & 4272   & 636   & 1312892  & 613   & 5571037   & 1213  \\
PIP (0.4)    & 3772   & 157   & 3776   & 344   & 3930   & 461   & 1309268  & 304   & 5490707   & 1156  \\
PIP (0.6)    & 3700   & 401   & 3710   & 590   & 3872   & 708   & 1237126  & 677   & 5489822   & 1343  \\
\bottomrule
\end{tabular}
\end{table}

\begin{table}[!h]
\centering
\caption{QAP Instances Group 3, $n < 30$}
\label{tab:qap_group3}
\small
\begin{tabular}{@{}l
  S[table-format=6.0]@{\hspace{1em}} S[table-format=3.0]
  S[table-format=6.0]@{\hspace{1em}} S[table-format=3.0]
  S[table-format=6.0]@{\hspace{1em}} S[table-format=3.0]
  S[table-format=3.0]@{\hspace{0.9em}} S[table-format=3.0]
  S[table-format=3.0]@{\hspace{0.9em}} S[table-format=3.0]@{}}
\toprule
& \multicolumn{2}{c}{bur26b}
& \multicolumn{2}{c}{bur26c}
& \multicolumn{2}{c}{bur26d} 
& \multicolumn{2}{c}{nug27} 
& \multicolumn{2}{c}{nug28} \\
\cmidrule(lr){2-3} \cmidrule(lr){4-5} \cmidrule(lr){6-7} \cmidrule(lr){8-9} \cmidrule(l){10-11}
{Method} & {Obj} & {Time} 
         & {Obj} & {Time} 
         & {Obj} & {Time} 
         & {Obj} & {Time} 
         & {Obj} & {Time} \\
\midrule
Best Known   & 3817852  &       & 5426795  &       & 3821225  &       & 5234   &       & 5166   &       \\
FMIP (10min / 20min) & 4051998  & 1213  & 6187881  & 1213  & 4013209  & 1213  & 6024   & 648   & 7054   & 618   \\
PIP (0.4)    & 3879427  & 985   & 5526051  & 730   & 3895655  & 1034  & 5622   & 511   & 5478   & 399   \\
PIP (0.6)    & 3857624  & 1234  & 5492074  & 1041  & 3883731  & 1219  & 5588   & 699   & 5478   & 525   \\
\bottomrule
\end{tabular}
\end{table}
\begin{table}[!h]
\centering
\caption{QAP Instances Group 4, $n \geq 30$}
\label{tab:qap_large}\small
\begin{tabular}{
  @{}l 
  S[table-format=6.0]@{\hspace{0.9em}} S[table-format=4.0] 
  S[table-format=6.0]@{\hspace{0.9em}} S[table-format=4.0] 
  S[table-format=6.0]@{\hspace{0.9em}} S[table-format=4.0] 
  S[table-format=6.0]@{\hspace{0.9em}} S[table-format=4.0] 
  S[table-format=6.0]@{\hspace{0.9em}} S[table-format=4.0]@{} 
}
\toprule
& \multicolumn{2}{c}{tho30}
& \multicolumn{2}{c}{tai30a} 
& \multicolumn{2}{c}{tai35a} 
& \multicolumn{2}{c}{tho40}
& \multicolumn{2}{c}{tai40a} \\
\cmidrule(lr){2-3} \cmidrule(lr){4-5} \cmidrule(lr){6-7} \cmidrule(lr){8-9} \cmidrule(l){10-11}
{Method} & {Obj} & {Time} 
         & {Obj} & {Time} 
         & {Obj} & {Time} 
         & {Obj} & {Time} 
         & {Obj} & {Time} \\
\midrule
Best Known     & 149936  &       & 1818146 &       & 2422002 &       & 240516  &       & 3139370 &       \\
FMIP (4hr / 8hr)     & 178866  & 14422 & 2073512 & 14477 & 2894058 & 14539 & 326506  & 14466 & 3836996 & 29040 \\
PIP (0.4)      & 164528  & 1267  & 1913228 & 4057  & 2613456 & 7619  & 269632  & 5351  & 3345830 & 25180 \\
PIP (0.6)      & 162876  & 4952  & 1909448 & 7683  & 2602810 & 11263 & 260936  & 8403  & 3326546 & 27629 \\
\bottomrule
\end{tabular}
\end{table}

\begin{table}[!h]
\centering
\caption{InvQP, $m = 200$, $n = 150$}
\label{tab:inverse_qp_200}
\small
\begin{tabular}{
  @{}l 
  S[table-format=-1.2] 
  S[table-format=4.0] 
  S[table-format=-1.2] 
  S[table-format=4.0] 
  S[table-format=-1.2] 
  S[table-format=4.0] 
  S[table-format=-1.2] 
  S[table-format=4.0] 
  S[table-format=-1.2] 
  S[table-format=4.0] 
  @{} 
}
\toprule
& \multicolumn{2}{c}{1} & \multicolumn{2}{c}{2} & \multicolumn{2}{c}{3} & \multicolumn{2}{c}{4} & \multicolumn{2}{c}{5} \\
\cmidrule(lr){2-3} \cmidrule(lr){4-5} \cmidrule(lr){6-7} \cmidrule(lr){8-9} \cmidrule(l){10-11}
{Method} & {Obj} & {Time} & {Obj} & {Time} & {Obj} & {Time} & {Obj} & {Time} & {Obj} & {Time} \\
\midrule
IPOPT & 353.55 & 4054 & 177.77 & 1431 & 183.65 & 1613 & 327.91 & 4283 & 295.90 & 4121 \\
FMIP (100min) & 159.68  & 6005  & 169.53  & 6005  & 163.51  & 6005 & 175.41 & 6005 & 162.91  & 6005  \\
FMIP-W (40min) & 159.90 & 6459 & 169.37 & 3841 & 163.51 & 4019 & 176.06 & 6692 & 163.46 & 6527 \\
PIP (0.4) & 160.07 & 4120 & 169.37 & 1451 & 166.25 & 1625 & 187.94 & 4322 & 164.29 & 4149 \\
PIP (0.6) & 159.15 & 4935 & 169.37 & 2058 & 163.51 & 2938 & 183.74 & 4800 & 161.02 & 4922 \\
PIP (0.8) & 159.15 & 6136 & 169.37 & 3259 & 163.51 & 4139 & 174.59 & 6602 & 161.02 & 6123 \\
\bottomrule
\end{tabular}

\end{table}

Last but not least, PIP also provides the best solution quality and efficiency in InvQP instances. Table \ref{tab:inverse_qp_200} shows that on all small InvQP instances, PIP ($0.6$) and PIP ($0.8$) obtain solutions with the highest quality in a shorter or comparable time than both FMIP and FMIP-W, providing significant improvement from the initial IPOPT solutions. While FMIP and FMIP-W also perform fairly well in small InvQP instances with 200 integer variables, similar to what we observe in small StQP instances, the performance gap becomes much more significant when the number of integer variables increases to $1000$.  
\gap

In the large InvQP instances, when the original LPCC expands to have $1000$ complementarity pairs, numerical issues start to occur due to the NLP solvers as mentioned in the implementation details, hence we opt for the alternative approach described previously to initialize PIP with FMIP (10min) solutions in order to maintain a fair comparison. Interestingly, as shown in 
Table \ref{tab:inverse_qp_1000}, despite being constrained to follow the same solution path as FMIP for the initial 10 minutes, PIP with \(p_{\max}\) values of 0.4, 0.6, or 0.8 consistently produces much better solutions with significantly lower objective values in shorter or comparable times across all instances. This further demonstrates the scalability and solution quality improvement capabilities of PIP.

\begin{table}[!h]
\centering
\caption{InvQP, $m=1000, n=750$}
\label{tab:inverse_qp_1000}
\setlength{\tabcolsep}{4.6pt} 
\small
\begin{tabular}{
  @{}l 
  S[table-format=-3.2] 
  S[table-format=3.0] 
  S[table-format=-1.1] 
  S[table-format=3.0] 
  S[table-format=-3.2] 
  S[table-format=3.0] 
  S[table-format=-1.1] 
  S[table-format=3.0] 
  S[table-format=-2.2] 
  S[table-format=4.0] 
  @{} 
}
\toprule
& \multicolumn{2}{c}{1} & \multicolumn{2}{c}{2} & \multicolumn{2}{c}{3} & \multicolumn{2}{c}{4} & \multicolumn{2}{c}{5} \\
\cmidrule(lr){2-3} \cmidrule(lr){4-5} \cmidrule(lr){6-7} \cmidrule(l){8-9} \cmidrule(l){10-11}
{Method} & {Obj} & {Time } & {Obj} & {Time } & {Obj} & {Time } & {Obj} & {Time } & {Obj} & {Time} \\
\midrule
FMIP (10min) & 11545.96 & 716 & 9611.32 & 718 & 678105.24 & 718 & 9601.07 & 717 & 14957.20 & 717 \\
FMIP (2.5hr) & 1614.41 & 9111 & 5702.10 & 9112 & 2225.22 & 9112 & 1678.53 & 9111 & 1649.75 & 9110 \\
PIP (0.4) & 707.56 & 5693 & 682.89 & 6303 & 670.77 & 6303 & 698.08 & 5086 & 719.38 & 5694 \\
PIP (0.6) & 707.56 & 6909 & 681.93 & 8129 & 669.68 & 8129 & 698.08 & 6302 & 719.38 & 6911 \\
PIP (0.8) & 706.63 & 8735 & 681.93 & 9346 & 669.68 & 9346 & 698.08 & 7520 & 719.38 & 8128 \\
\bottomrule
\end{tabular}

\end{table}

\gap

\noindent $\bullet $ {\bf InvAVI results.}
PIP produces solutions with better objective values in shorter
times on all instances with 150 or more complementarity pairs 
shown in Tables \ref{tab:InvAVI_150} through  
\ref{tab:InvAVI_350}. On the small instances with 100 
complementarity pairs in Table~\ref{tab:InvAVI_100}, PIP achieves 
global optimal solutions fast in 2 out of 5 instances, while 
providing solid improvement from initial solutions on all 
instances. 

\gap

\noindent The results on the InvAVI instances further expose the
inability of scalability for FMIP and FMIP-W approaches, 
as FMIP (1hr) is not able to produce any feasible solution on any 
instance when the number of complementarity pairs exceeds 150. 
FMIP-W also suffers from scalability with deficient improvement 
over initial solutions, especially on larger instances in Table 
\ref{tab:InvAVI_300} and  \ref{tab:InvAVI_350}. 
In contrast, PIP shows a much better scalability with 
consistently better improvements in solution quality on larger 
instances.

\gap

\noindent Interestingly, the performance of PIP with lower to 
medium $p_{\max}$ values (0.4 and 0.6) are already better than 
FMIP and FMIP-W on most instances, which might be due to the high 
complexity of the problem. 
Although PIP ($0.8$) remains the best 
approach in solution quality on InvAVI instances, one may 
find the PIP ($0.4$) solutions satisfactory while saving 
more computation time. In Table \ref{tab:InvAVI_250} to \ref{tab:InvAVI_350}, we omit the PIP ($0.6$) and ($0.8$) results since in most instances they do not provide significant advantage over PIP~($0.4$).

\begin{table}[h]
\centering
\caption{InvAVI, \(m=100,\;n=25\)}
\label{tab:InvAVI_100}
\small
\begin{tabular}{
  @{}l
  S[table-format=-2.2]
  S[table-format=4.0]
  S[table-format=-2.2]
  S[table-format=4.0]
  S[table-format=-2.2]
  S[table-format=4.0]
  S[table-format=-2.2]
  S[table-format=4.0]
  S[table-format=-2.2]
  S[table-format=4.0]
  @{}
}
\toprule
& \multicolumn{2}{c}{1} & \multicolumn{2}{c}{2} & \multicolumn{2}{c}{3} & \multicolumn{2}{c}{4} & \multicolumn{2}{c}{5} \\
\cmidrule(lr){2-3} \cmidrule(lr){4-5} \cmidrule(lr){6-7} \cmidrule(lr){8-9} \cmidrule(l){10-11}
{Method} & {Obj} & {Time} & {Obj} & {Time} & {Obj} & {Time} & {Obj} & {Time} & {Obj} & {Time} \\
\midrule
IPOPT         & -6.11  & 1   & -11.45 & 1   & -9.67  & 1   & -9.37  & 1   & -8.28  & 1 \\
FMIP (Opt)    & -7.47  & 147 & -15.25 & 253 & -12.75 & 24  & -12.42 & 39  & -10.74 & 20 \\
FMIP-W (Opt)  & -7.47  & 42  & -15.25 & 89  & -12.75 & 28  & -12.42 & 98  & -10.74 & 23 \\
PIP (0.4)     & -7.23  & 4   & -13.38 & 7   & -12.10 & 5   & -9.37  & 4   & -9.88  & 5 \\
PIP (0.6)     & -7.23  & 6   & -13.38 & 16  & -12.10 & 7   & -9.37  & 7   & -10.25 & 10 \\
PIP (0.8)     & -7.23  & 14  & -14.58 & 70  & -12.75 & 16  & -12.42 & 34  & -10.35 & 27 \\
\bottomrule
\end{tabular}
\end{table}
\begin{table}[h]
\centering
\caption{InvAVI, \(m=150,\;n=38\)}
\label{tab:InvAVI_150}\small
\begin{tabular}{
  @{}l
  S[table-format=-2.2]
  S[table-format=4.0]
  S[table-format=-2.2]
  S[table-format=4.0]
  S[table-format=-2.2]
  S[table-format=4.0]
  S[table-format=-2.2]
  S[table-format=4.0]
  S[table-format=-2.2]
  S[table-format=4.0]
  @{}
}
\toprule
& \multicolumn{2}{c}{1} & \multicolumn{2}{c}{2} & \multicolumn{2}{c}{3} & \multicolumn{2}{c}{4} & \multicolumn{2}{c}{5} \\
\cmidrule(lr){2-3} \cmidrule(lr){4-5} \cmidrule(lr){6-7} \cmidrule(lr){8-9} \cmidrule(lr){10-11}
{Method} & {Obj} & {Time} & {Obj} & {Time} & {Obj} & {Time} & {Obj} & {Time} & {Obj} & {Time} \\
\midrule
IPOPT         & -14.09 & 2   & -12.63 & 3   & -19.65 & 3   & -15.33 & 3   & -9.78  & 3 \\
FMIP (1hr)    & NA     & 3606& NA     & 3607& NA     & 3606& NA     & 3606& NA     & 3606 \\
FMIP-W (1hr)  & -15.43 & 3610& -14.64 & 3610& -20.68 & 3610& -16.82 & 3609& -10.86 & 3609 \\
PIP (0.4)     & -17.75 & 186 & -14.64 & 142 & -20.84 & 20  & -18.40 & 31  & -14.06 & 65 \\
PIP (0.6)     & -17.75 & 1386& -15.78 & 1089& -24.25 & 1292& -18.40 & 375 & -14.67 & 705 \\
PIP (0.8)     & -17.92 & 3788& -15.78 & 2290& -24.25 & 2493& -18.40 & 1576& -14.67 & 1906 \\
\bottomrule
\end{tabular}

\end{table}
\begin{table}[h]
\centering
\caption{InvAVI, \(m=250,\;n=63\)}
\label{tab:InvAVI_250}\small
\begin{tabular}{
  @{}l
  S[table-format=-2.2]
  S[table-format=4.0]
  S[table-format=-2.2]
  S[table-format=4.0]
  S[table-format=-2.2]
  S[table-format=4.0]
  S[table-format=-2.2]
  S[table-format=4.0]
  S[table-format=-2.2]
  S[table-format=4.0]
  @{}
}
\toprule
& \multicolumn{2}{c}{1} & \multicolumn{2}{c}{2} & \multicolumn{2}{c}{3} & \multicolumn{2}{c}{4} & \multicolumn{2}{c}{5} \\
\cmidrule(lr){2-3} \cmidrule(lr){4-5} \cmidrule(lr){6-7} \cmidrule(lr){8-9} \cmidrule(lr){10-11}
{Method} & {Obj} & {Time} & {Obj} & {Time} & {Obj} & {Time} & {Obj} & {Time} & {Obj} & {Time} \\
\midrule
IPOPT         & -20.73 & 7   & -13.23 & 7   & -16.21 & 7   & -24.77 & 7   & -18.85 & 7 \\
FMIP (1hr)    & NA     & 3610& NA     & 3610& NA     & 3610& NA     & 3610& NA     & 3610 \\
FMIP-W (1hr)  & -21.67 & 3617& -13.23 & 3617& -16.21 & 3617& -24.77 & 3617& -19.24 & 3617 \\
PIP (0.4)     & -22.60 & 1658& -15.61 & 1344& -18.37 & 1987& -26.60 & 3862& -19.24 & 1406 \\
\bottomrule
\end{tabular}
\end{table}
\begin{table}[h]
\centering
\caption{InvAVI, \(m=300,\;n=75\)}
\label{tab:InvAVI_300}\small
\begin{tabular}{
  @{}l
  S[table-format=-2.2]
  S[table-format=4.0]
  S[table-format=-2.2]
  S[table-format=4.0]
  S[table-format=-2.2]
  S[table-format=4.0]
  S[table-format=-2.2]
  S[table-format=4.0]
  S[table-format=-2.2]
  S[table-format=4.0]
  @{}
}
\toprule
& \multicolumn{2}{c}{1} & \multicolumn{2}{c}{2} & \multicolumn{2}{c}{3} & \multicolumn{2}{c}{4} & \multicolumn{2}{c}{5} \\
\cmidrule(lr){2-3} \cmidrule(lr){4-5} \cmidrule(lr){6-7} \cmidrule(lr){8-9} \cmidrule(lr){10-11}
{Method} & {Obj} & {Time} & {Obj} & {Time} & {Obj} & {Time} & {Obj} & {Time} & {Obj} & {Time} \\
\midrule
IPOPT         & -23.63 & 10  & -27.85 & 11  & -29.08 & 10  & -23.54 & 10  & -19.25 & 10 \\
FMIP (1hr)    & NA     & 3614& NA     & 3614& NA     & 3614& NA     & 3614& NA     & 3614 \\
FMIP-W (1hr)  & -23.63 & 3624& -28.05 & 3625& -29.08 & 3624& -23.56 & 3624& -19.43 & 3624 \\
PIP (0.4)     & -26.36 & 2303& -31.04 & 2685& -32.56 & 2852& -25.15 & 1326& -20.30 & 2243 \\
\bottomrule
\end{tabular}
\end{table}
\begin{table}[h]
\centering
\caption{InvAVI, \(m=350,\;n=88\)}
\label{tab:InvAVI_350}\small
\begin{tabular}{
  @{}l
  S[table-format=-2.2]
  S[table-format=4.0]
  S[table-format=-2.2]
  S[table-format=4.0]
  S[table-format=-2.2]
  S[table-format=4.0]
  S[table-format=-2.2]
  S[table-format=4.0]
  S[table-format=-2.2]
  S[table-format=4.0]
  @{}
}
\toprule
& \multicolumn{2}{c}{1} & \multicolumn{2}{c}{2} & \multicolumn{2}{c}{3} & \multicolumn{2}{c}{4} & \multicolumn{2}{c}{5} \\
\cmidrule(lr){2-3} \cmidrule(lr){4-5} \cmidrule(lr){6-7} \cmidrule(lr){8-9} \cmidrule(lr){10-11}
{Method} & {Obj} & {Time} & {Obj} & {Time} & {Obj} & {Time} & {Obj} & {Time} & {Obj} & {Time} \\
\midrule
IPOPT         & -17.43 & 15  & -29.09 & 16  & -35.83 & 16  & -29.11 & 16  & -28.29 & 15 \\
FMIP (1hr)    & NA     & 3618& NA     & 3618& NA     & 3618& NA     & 3618& NA     & 3618 \\
FMIP-W (1hr)  & -17.44 & 3633& -29.11 & 3634& -36.08 & 3634& -29.11 & 3634& -28.29 & 3634 \\
PIP (0.4)     & -19.47 & 2435& -30.01 & 2436& -36.18 & 2436& -30.80 & 3036& -29.66 & 2284 \\
\bottomrule
\end{tabular}
\end{table}

\subsection{Summary of numerical results}
\label{subsec:summary}

\noindent The numerical results reported in the
last subsection demonstrate the overall superior 
performance of the proposed PIP method over FMIP, 
FMIP-W, and global solvers on the StQP tested instances;
most worthy of note is the improvement of the 
solution quality for indefinite QPs and LPCCs from
NLP solvers.

\gap

\noindent For the StQP instances, PIP consistently produces high-quality solutions within comparable computational times. In smaller instances, PIP (0.9) matches the global optimal solutions. And as 
the problem size increases, PIP (0.8) provides superior or competitive solutions more rapidly than the other methods.

\gap

\noindent For the QAP instances, PIP (0.6) is the most effective method in terms of solution quality and computation time. On smaller instances, PIP achieves the best known or nearby solutions quickly. For larger instances, PIP produces solutions close to the best known results much faster than the full MILP approach. PIP (0.4) provides good balance between solution quality and computation speed.

\gap

\noindent For the InvQP instances, PIP (0.6) and PIP (0.8) excel in solution quality and computational speed on smaller instances. On larger instances, PIP demonstrates notable scalability and provides significant improvement in solution quality effectively.

\gap
\noindent For the InvAVI instances, PIP provides 
notable solution quality improvement with PIP (0.4) 
showcasing strong scalability and efficiency over the
FMIP formulations.

\section{Conclusion}

\noindent   In this paper, we have proposed a novel
progressive integer programming (PIP) method for solving 
and enhancing solutions in nonconvex quadratic programs 
and linear programs with complementarity constraints. 
We show that the solution produced by the new method has 
desirable analytical properties that cannot be claimed 
by the traditional nonlinear programming 
algorithms.  For the nonconvex QP, we 
prove the local optimality of the computed PIP solutions 
(assuming exact computations)
for the associated QP-induced LPCC and 
reveal the connections between suboptimal KKT triples and 
solutions to the original quadratic program.
For a general LPCC, the same local 
optimality claim is valid.  Through large amounts of 
numerical experiments, 
we demonstrate that PIP is a powerful technique for 
escaping from inferior solutions computed by an NLP 
algorithm.  In particular, when compared with full 
mixed-integer approaches, the proposed method exhibits 
significant advantages in scalability and flexibility 
over the instances tested.

Originated from \cite{FangLiuPang2024}, the PIP 
idea has the potential for solving wide classes of 
continuous
optimization problems fused with discrete indicators 
that can capture logical conditions.
Complementarity is an example of such problems; 
the class of
{\sl Heaviside composite functions} that has many 
applications of its own is another.  The latter are 
compositions of functions (possibly nonsmooth) with the
univariate Heaviside functions which themselves
are the indicators of intervals on the real line.  A 
comprehensive study 
of the optimization of such functions is presently ongoing
with many results already available; PIP provides a key
approach for solving the resulting optimization problems.

\end{document}